\documentclass[table]{amsart}

\makeatletter
\newcommand{\mylabel}[2]{#2\def\@currentlabel{#2}\label{#1}}
\makeatother
\usepackage{tikz-cd}
\usepackage{amssymb,mathrsfs,mathtools}
\usepackage{comment}
\usepackage{float}
\usepackage{longtable} 
\usepackage{pdflscape} 
\usepackage{orcidlink}
\usepackage{booktabs}
\usepackage{hyperref}
\usepackage{geometry}
\definecolor{vegasgold}{rgb}{0.77, 0.7, 0.35}
\definecolor{darkgoldenrod}{rgb}{0.72, 0.53, 0.04}
\definecolor{gold(metallic)}{rgb}{0.83, 0.69, 0.22}
\hypersetup{
 colorlinks=true,
 linkcolor=darkgoldenrod,
 filecolor=brown,      
 urlcolor=gold(metallic),
 citecolor=darkgoldenrod,
 }

\usepackage[all,cmtip]{xy}

\newtheorem{theorem}{Theorem}[section]
\newtheorem{lemma}[theorem]{Lemma}
\newtheorem{ass}[theorem]{Assumption}
\newtheorem{definition}[theorem]{Definition}

\newtheorem{remark}[theorem]{Remark}

\newtheorem{proposition}[theorem]{Proposition}

\newcommand{\llbracket}{[\![}
\newcommand{\rrbracket}{]\!]}

\newcommand{\Zp}{\mathbb{Z}_p}

\newcommand{\ord}{\mathrm{ord}}

\newcommand{\Z}{\mathbb{Z}}

\newcommand{\Q}{\mathbb{Q}}
\newcommand{\F}{\mathbb{F}}

\newcommand{\cO}{\mathcal{O}}

\newcommand{\Hom}{\mathrm{Hom}}

\newcommand{\Sel}{\mathrm{Sel}}

\usepackage[OT2,T1]{fontenc}
\DeclareSymbolFont{cyrletters}{OT2}{wncyr}{m}{n}
\DeclareMathSymbol{\Sha}{\mathalpha}{cyrletters}{"58}

\newcommand{\op}[1]{\operatorname{#1}}

\newcommand\smallmtx[4] { {\begin{psmallmatrix}
 #1 & #2 \\
 #3 & #4 \\
 \end{psmallmatrix} } }

\numberwithin{equation}{section}

\begin{document}

\title{An analogue of Kida's formula for Mazur--Tate elements}
\author[N.~Pratap]{Naman Pratap\, \orcidlink{0009-0004-3901-8141}}
\address[Pratap]{Centre for Mathematical Sciences, Wilberforce Road,
Cambridge, CB3 0WA, United Kingdom}
\email{namanpratap2001@gmail.com}
\author[A.~Ray]{Anwesh Ray\, \orcidlink{0000-0001-6946-1559}}\thanks{Corresponding author: Anwesh Ray.}
\address[Ray]{Chennai Mathematical Institute, H1, SIPCOT IT Park, Kelambakkam, Siruseri, Tamil Nadu 603103, India}
\email{ar2222@cornell.edu}


\maketitle
\begin{abstract}
We prove an analogue of Kida’s formula for the Iwasawa invariants of the Mazur--Tate elements attached to elliptic curves over~$\Q$. Let $p$ be an odd prime and let $L/K$ be a Galois extension of abelian number fields with $p$-power Galois group. For an elliptic curve $E/\Q$, we study the Mazur--Tate elements over the finite layers of the cyclotomic $\Z_p$-extensions of $K$ and $L$. We show that the vanishing of the $\mu$-invariant is preserved in the extension: if the level-$n$ Mazur--Tate element over $K$ has $\mu=0$, then the corresponding element over $L$ also has $\mu=0$. Moreover, the associated $\lambda$-invariants satisfy an explicit transition formula. This parallels the work of Hachimori--Matsuno on Selmer groups and of Matsuno on $p$-adic $L$-functions. As an application, we obtain an analogue of Kida’s formula for the analytic Iwasawa invariants associated to Pollack’s signed $p$-adic $L$-functions. Since our results apply to elliptic curves with any reduction type at $p$ under mild hypotheses, including those with additive reduction, we also obtain a Kida-type formula for the $p$-adic $L$-functions constructed by Delbourgo for elliptic curves with unstable additive reduction. In particular, because Mazur--Tate elements approximate $p$-adic $L$-functions in the limit, our results simultaneously unify all previously known cases of Kida’s formula for analytic Iwasawa invariants.
\end{abstract}

\section{Introduction}
\subsection{An arithmetic analogue of the Riemann--Hurwitz formula}
\par The classical Riemann--Hurwitz formula gives a precise relation between the genera of two compact Riemann surfaces $X$ and $Y$, where $\varphi: Y\rightarrow X$ is a branched cover. Let $g_Y$ (resp. $g_X$) denote the genus of $X$ (resp. $Y$). Then one has that
\[
2g_Y - 2 \;=\; \deg(\varphi)\,(2g_X - 2)\;+\; \sum_{P\in Y} (e(P)-1),
\]
where the sum runs over all points $P\in Y$, with $e(P)$ the ramification index of $\varphi$ at $P$. For an unramified cover, the genera are related by \[2g_Y-2=\deg(\varphi)(2g_X-2).\] The presence of ramification introduces further correction terms in the formula. The quantity $(e(P)-1)$ depends only on the point $P$ and is thought of as a local term, while the term $\deg(\varphi)(2g_X-2)$ is global in nature.

\par A remarkable analogue of this relation exists in arithmetic. In place of coverings of Riemann surfaces one considers extensions of number fields. The Iwasawa $\lambda$-invariant plays an analogous role to the genus of a Riemann surface. Let $K$ be a number field and $p$ be a prime number. Denote by $K(\mu_{p^\infty})$ the infinite Galois extension extension of $K$ obtained by adjoining all $p$-power roots of unity. Let $\Z_p$ denote the ring of $p$-adic integers. The cyclotomic $\Z_p$-extension is the unique Galois extension $K_\infty/K$ contained in $K(\mu_{p^\infty})$ with Galois group $\op{Gal}(K_\infty/K)$ isomorphic to $\Z_p$. For each integer $n\ge 1$, write $K_{(n)}$ for the unique intermediate extension of $K_\infty/K$ of degree $p^n$ over $K$.
\par In \cite{KIDA1980519}, Kida proved an analogue of the Riemann--Hurwitz genus formula. Assume that $K$ is a CM field and let $K_{(n)}^+$ be the maximal totally real subfield of $K_{(n)}$. We set \[h_n:=h(K_{(n)}),\quad h_n^+:=h(K_{(n)}^+)\quad \text{and}\quad h_n^-:=h_n/h_n^+.\] Iwasawa showed (see \cite{Iwasawa:1959} and \cite{iwasawa1973zl}) that there exist nonnegative integers $ \mu_{K}^{-},\lambda_K^{-}\in \Z_{\geq 0}$ and $\nu_K^{-} \in \Z$ such that for $n$ large, one has
\begin{equation} \label{Iwasawa}
{\rm ord}_{p}(h_{n}^{-}) = \mu_K^{-}  \ell^{n} + \lambda_K^{-} n + \nu_K^{-}, 
\end{equation}
where ${\rm ord}_{p}$ is the usual $p$-adic valuation on $\mathbb{Q}$. Let $L/K$ be a Galois extension of CM number fields for which ${\rm Gal}(L/K)$ is a $p$-group, and assume for simplicity that $p> 2$ and that $K$ contains the $p$-th roots of unity.  Kida showed that if $\mu_{K}^{-} = 0$, then $\mu_{L}^{-} = 0$, and 
\begin{equation} \label{kida}
2\lambda_{L}^{-} - 2 = [L_{\infty}:K_{\infty}]( 2\lambda_{K}^{-} - 2) + \sum_{w}(e_{w} - 1),
\end{equation}
where the sum is over all nonarchimedean places of $L_{\infty}$ that split completely in $L_{\infty}/L_{\infty}^{+}$ and that do not lie above $p$. Gras \cite{Gras} had obtained similar transition formulae for the Iwasawa invariants of Kubota--Leopoldt $p$-adic $L$-functions, which were later extended to the case of CM fields by Sinnott in \cite{Sinnott}.  

\subsection{Iwasawa theory of elliptic curves}
In number theory, many Galois theoretic structures arise from geometry and analysis such as elliptic curves and modular forms. Their associated Galois representations and L-functions lead to the study of Iwasawa theoretic phenomena. Mazur~\cite{Mazur1972} inititated the Iwasawa theory for $p$-ordinary abelian varieties. Let $p$ be a prime number, $K$ a number field and $A_{/K}$ be an abelian variety which is ordinary at all primes $v|p$ of $K$. Mazur studied the structure and growth of the $p$-primary Selmer group of $A$ along the cyclotomic $\mathbb{Z}_p$-extension of $K$. In the case of an elliptic curve $E/\mathbb{Q}$ and a prime $p$ of good ordinary reduction for $E$, Kato~\cite{kato1} and Rubin~\cite{Rubin2} proved that the $p$-primary Selmer group $\mathrm{Sel}_{p^\infty}(E/\mathbb{Q}_\infty)$ is a cotorsion module over the Iwasawa algebra $\Lambda=\mathbb{Z}_p[[\Gamma]]$ with $\Gamma:=\operatorname{Gal}(\mathbb{Q}_\infty/\mathbb{Q})$. Consequently, the Pontryagin dual of the Selmer group admits a well-defined characteristic ideal. Choose a topological generator $\gamma\in \Gamma$ and set $T:=\gamma-1$. Then $\Lambda$ is identified with the formal power series ring $\Z_p\llbracket T\rrbracket$. The characteristic ideal is principal, and a choice of generator $f_{\op{Iw}}(T)$ is the \emph{Iwasawa polynomial}. By the Weierstrass preparation theorem, $f_{\op{Iw}}(T)=p^{\mu}P(T)u(T)$ where $\mu\in \Z_{\geq 0}$, $P(T)$ is a distinguished polynomial and $u(T)$ is a unit in $\Lambda$. The quantity $\mu$ is the $\mu$-invariant and the $\lambda$-invariant is defined to be the degree of $P(T)$.
\par A key ingredient in Mazur’s foundational study of Iwasawa theory for elliptic curves is the \emph{control theorem} for the Selmer group. This result compares the Selmer groups over the finite layers $\Q_{(n)}$ of the cyclotomic $\Z_p$-extension with the Selmer group over the full infinite extension $\Q_\infty$. It asserts that, up to a finite and uniformly bounded error term, the $p$-primary Selmer group over $\Q_{(n)}$ is obtained by taking $\Gamma^{p^n}$-invariants of the Selmer group over $\Q_\infty$. Consequently, Mazur’s control theorem yields an Iwasawa-theoretic formula describing the asymptotic growth of the $p$-primary part of the Tate--Shafarevich group of $E/\Q_{(n)}$ as $n\to\infty$, directly paralleling Iwasawa’s classical formula for $p$-primary class groups in cyclotomic towers. More generally, the same formalism extends to modular forms via their associated $p$-adic Galois representations, with the case of elliptic curves arising from weight two newforms as a special instance.

\par On the analytic side, one may study the variation of the special values of the complex $L$-function of $E$ along the layers of the cyclotomic extension. Using the modular symbols attached to the weight two cuspidal eigenform corresponding to $E$ by modularity, Mazur and Swinnerton-Dyer constructed in \cite{MSD74} a $p$-adic $L$-function for $E$ when $p$ is a prime of good ordinary reduction. This $p$-adic analytic function interpolates the critical values $L(E,\chi,1)$ as $\chi$ ranges over finite-order characters of $\operatorname{Gal}(\Q_{(n)}/\Q)$. Moreover, the $p$-adic $L$-function naturally defines an element of the Iwasawa algebra $\Lambda \cong \Z_p\llbracket T\rrbracket$, thereby giving rise to \emph{analytic Iwasawa invariants}. The Iwasawa Main Conjecture predicts a precise equality between these analytic invariants associated with the $p$-adic $L$-function and the corresponding algebraic invariants arising from the structure of the Selmer group.

\par For primes of supersingular reduction, the situation is more subtle since in this setting the Selmer group over the cyclotomic $\Z_p$-extension is no longer cotorsion over the Iwasawa algebra. Kobayashi \cite{kobayashi} introduced the \emph{plus and minus Selmer groups} $\op{Sel}^{\pm}(E/\Q_{(n)})$, which refine the usual Selmer group by imposing conditions on the formal group at $p$. Correspondingly, Pollack constructed \emph{signed $p$-adic $L$-functions} $L_p^\pm(E)$ that interpolate critical $L$-values twisted by characters of $\op{Gal}(\Q_{(n)}/\Q)$, but with plus and minus signs reflecting the local conditions at $p$. 
\par In \cite{MT}, Mazur and Tate introduced modular elements $\theta_M(f)$ for integers $M$ coprime to $p$, now known as \emph{Mazur--Tate elements}, which lie in the group ring $\Z[\op{Gal}(\Q(\mu_M)/\Q)]$ and encode special values of the complex $L$-function of a cuspidal eigenform $f$. Of particular interest for Iwasawa-theoretic applications are the $p$-adic Mazur--Tate elements, which may be viewed inside the group rings $\Lambda_n := \Z_p[\op{Gal}(\Q_{(n)}/\Q)]$. Each $\Lambda_n$ is naturally identified with a quotient of the Iwasawa algebra $\Lambda$, and under this identification the Mazur--Tate elements inherit associated $\mu$- and $\lambda$-invariants. These provide a finite-level analogue of the analytic Iwasawa invariants and play an important role in understanding the variation of $L$-values and Selmer groups across the cyclotomic tower. For elliptic curves with additive reduction satisfying further conditions, Delbourgo \cite{del-compositio} proved that the Selmer group is cotorsion over the Iwasawa algebra and constructed $p$-adic \textit{L}-functions, formulating an analogue of the Main conjecture.

When $p$ is a prime of good ordinary reduction, Mazur--Tate elements approximate the $p$-adic $L$-function, which can be realized as the inverse limit of ($p$-stabilized) Mazur--Tate elements under the norm maps $\Lambda_{n}\rightarrow \Lambda_{n-1}$. Analogous to the Iwasawa Main conjecture, these finite-layer analogues are expected to encode the structure of the \emph{initial Fitting ideal} of the Selmer group over $\Q_{(n)}$. This constitutes the \emph{refined main conjecture}, studied in \cite{kimkurihara} for elliptic curves with both ordinary and supersingular reduction at $p$. Mazur--Tate elements are also explicitly related to Kato's zeta elements via a pairing defined by Kurihara \cite[Lemma 7.2]{kuri2002}, which is a key ingredient in the proof of the refined conjectures for supersingular primes.
\subsection{Kida-type formulae for elliptic curves and modular forms: Known results}

In \cite{Hachimori1999581}, Hachimori and Matsuno established an analogue of Kida's classical formula for the $\mu$- and $\lambda$-invariants of $p$-primary Selmer groups of elliptic curves. Let $E/\Q$ be an elliptic curve with good ordinary reduction at an odd prime $p$, and let $L/\Q$ be a $p$-cyclic extension. Assuming $\mu\bigl(\Sel_{p^\infty}(E/\Q_\infty)\bigr)=0$, the Selmer group $\Sel_{p^\infty}(E/L_\infty)$ is cotorsion over the Iwasawa algebra, with $\mu\bigl(\Sel_{p^\infty}(E/L_\infty)\bigr)=0$. Moreover, the $\lambda$-invariants satisfy
\[
\lambda\bigl(\Sel_{p^\infty}(E/L_\infty)\bigr) = [L_\infty:\Q_\infty] \, \lambda\bigl(\Sel_{p^\infty}(E/\Q_\infty)\bigr) + \sum_{\tilde{w} \in P_1} \bigl(e(\tilde{w}) - 1\bigr) + 2 \sum_{\tilde{w} \in P_2} \bigl(e(\tilde{w}) - 1\bigr),
\]
where $P_1$ (resp. $P_2$) denotes the set of primes of $L_\infty$ above which $E$ has split multiplicative reduction (resp. has a $p$-torsion point locally), and $e(\tilde{w})$ is the ramification index of $\tilde{w}$ over the prime of $\Q_\infty$ below it. This result shows that the vanishing of the $\mu$-invariant is preserved in $p$-cyclic extensions and provides an explicit formula controlling the change in $\lambda$-invariants. Analogous results hold for elliptic curves with multiplicative reduction.

\par The analytic counterpart of this formula was established in \cite{MATSUNO200080} for the $p$-adic $L$-function of elliptic curves over $\Q$ at primes of good ordinary or multiplicative reduction. These Kida-type formulae for the analytic Iwasawa invariants match the algebraic invariants of the $p$-primary Selmer group, as predicted by the Iwasawa Main conjecture over $p$-extensions. In \cite{PWKida}, Pollack and Weston extended these results to modular
eigenforms and, on the algebraic side, to a broad class of ordinary Galois
representations, including representations arising from Hilbert modular forms
under suitable ordinarity hypotheses. They also treated supersingular primes for
elliptic curves. Their method uses congruences between modular forms, and hence
between the associated residual Galois representations, to compare the relevant
Iwasawa invariants.

\subsection{Main result}

In this article, we extend the methods of \cite{Hachimori1999581}, \cite{MATSUNO200080}, and \cite{Sinnott} to obtain an analogue of Kida's formula for the Iwasawa invariants of Mazur--Tate elements attached to elliptic curves. Let $E$ be an elliptic curve over $\Q$ and $p$ be an odd prime number. For $n\in \Z_{\geq 1}$, we consider Mazur--Tate elements $\Theta_n(E/K)$ over abelian extensions $K/\Q$. By the modularity theorem, there exists a weight $2$ newform $f \in S_2(\Gamma_0(N_E))$ associated to $E$, whose $L$-function agrees with that of $E$. In particular, for any number field $K$ and any finite order character $\chi$, we have
\[
L(E/K,\chi,s) = L(f/K,\chi,s).
\]
The Mazur--Tate elements are constructed to interpolate the special values $L(f/K,\chi,1)$ for primitive characters $\chi$ of $\op{Gal}(K_{(n)}/K)$. 
\par Let $S_{\op{add}}$ be the set of rational primes $\ell$ at which $E$ has additive reduction. We introduce the following hypothesis for a given number field $F$:
\begin{ass}[\textbf{Add}]
    $E$ has additive reduction at any prime of $F$ lying above $S_{\op{add}}$.
\end{ass}
\noindent Let $L/K$ be a $p$-extension with both $L$ and $K$ abelian over $\Q$ and $n\geq 1$. We define sets of primes $P_1$ and $P_2$ of $L_{(n)}$ as follows:
    \[\begin{split}&P_1 := \{w : w\nmid p,\; E \text{ has split multiplicative reduction at } w\},\\
        &P_2 := \{w : w\nmid p,\; E \text{ has good reduction at } w,\; E(L_{(n),w})[p]\neq 0\}.
        \end{split}
    \]
    Let $n_L,n_K$ be integers with \[L\cap \Q_\infty=\Q_{(n_L)}\quad\text{and}\quad K\cap \Q_\infty=\Q_{(n_K)}.\] In other words, $L_{(n)}=L\Q_{(n+n_L)}$ and $K_{(n)}=K\Q_{(n+n_K)}$. In particular, $L_{(n)}$ contains $K_{(n+n_L-n_K)}$. Denote by $e_{L_{(n)}/K_{(n+n_L-n_K)}}(w)$ (resp. $f_{L_{(n)}/K_{(n+n_L-n_K)}}(w)$) the ramification index (resp. inertial degree) of $w$ in $L_{(n)}/K_{(n+n_L-n_K)}$. Note that if $n\geq \op{ord}_p([L:K])-n_L$, the extension $L_{(n)}/K_{(n+n_L-n_K)}$ is totally ramified for any prime $w\nmid p$ of $L_{(n)}$. In other words, for $n\geq \op{ord}_p([L:K])-n_L$, \[f_{L_{(n)}/K_{(n+n_L-n_K)}}(w)=1.\]

\begin{ass}[\textbf{K-p}]
    One of the following conditions are satisfied:
    \begin{enumerate}
        \item $p\nmid [K:\Q]$, 
        \item $p\mid[K:\Q]$ and $n\geq \op{ord}_p([L:\Q])$.
    \end{enumerate}
\end{ass}
\noindent The main result of this paper is the following Kida-type formula for Mazur--Tate elements.  
\begin{theorem}[Kida's formula for elliptic curves]\label{thm: kida formula!}
     Let $n\geq 1$ and assume that the conditions \textbf{(Add)} and \textbf{(K-p)} are satisfied. If
    \[
        \mu(\Theta_{\,n+n_L-n_K}(E/K))=0,
    \]
    then $\mu(\Theta_n(E/L))=0$, and
    \[
    \begin{aligned}
        \lambda(\Theta_n(E/L)) 
        &= [L_\infty : K_\infty] \, \lambda(\Theta_{\,n+n_L-n_K}(E/K))
        + \sum_{w \in P_1} f_{L_{(n)}/K_{(n+n_L-n_K)}}(w)\bigl(e_{L_{(n)}/K_{(n+n_L-n_K)}}(w)-1\bigr) \\
        &\quad + 2 \sum_{w \in P_2} f_{L_{(n)}/K_{(n+n_L-n_K)}}(w)\bigl(e_{L_{(n)}/K_{(n+n_L-n_K)}}(w)-1\bigr).
    \end{aligned}
    \]
\end{theorem}

The theorem above demonstrates that such Kida-type formulae can indeed be established for Mazur--Tate elements, even for modular forms whose $p$-adic $L$-functions are not known to exist or be bounded. In particular, these results hold without imposing any hypotheses on the local behaviour of $E$ at $p$.
\par For elliptic curves with semistable reduction at $p>2$, one may attach \(p\)-adic \(L\)-functions not only over \(\Q\) but also over abelian extensions \(K/\Q\) obtained by combining the \(p\)-adic \(L\)-functions twisted by the characters of \(\op{Gal}(K/\Q)\), cf.~\eqref{defn of Lp(E/K)}. The same construction, when applied to the signed \(p\)-adic \(L\)-functions associated to \(E\), likewise produces signed \(p\)-adic \(L\)-functions over \(K\). In the semistable case, the Iwasawa invariants $\mu$ and $\lambda$ of $\Theta_n(E/\Q)$ for $n\gg 0$ agrees with the $\mu$ and $\lambda$ invariants of the $p$-adic $L$-function of $E$. Thus, the result above implies Matsuno's main result \cite{MATSUNO200080}. For elliptic curves $E_{/\Q}$ with good supersingular reduction at $p$ satisfying $a_p(E)=0$, we obtain an analogue of Kida's formula for the analytic Iwasawa invariants associated to signed $p$-adic \textit{L}-functions (see Theorem \ref{kida for signed p-adic $L$-functions} for further details). Similar formulae for the algebraic Iwasawa invariants for signed Selmer groups have been proven by Hatley and Lei \cite{hatleylei}, and there are further refinements due to Forr\'as and M\"uller \cite{forrasmuller}.
\par In the additive reduction case, an analogue of Kida’s formula was established by the second-named author and Shingavekar \cite{RSadditivekida} under the potentially good ordinary condition of Delbourgo \cite{del-compositio}, which requires that \(E\) acquire good ordinary reduction at \(p\) over an extension contained in \(\Q(\mu_p)\). Under this hypothesis, Delbourgo constructed a \(p\)-adic \(L\)-function for \(E\) by relating the newform \(f_E\) to an ordinary weight~\(2\) newform \(\tilde f\) of lower level obtained via twisting by a power of the Teichmüller character. This yields a bounded \(p\)-adic \(L\)-function to which Matsuno’s method applies directly, giving a Kida-type formula in this setting (see Theorem \ref{additive reduction kida formula}).

\subsection{Outlook}
These formulae align well with the philosophy of the \emph{refined main conjectures} as discussed in \cite{kimkurihara}. Namely, we get a Kida formula for the $\lambda$-invariant at every finite layer of the cyclotomic $\Zp$-extension, indicating that an algebraic analogue of this would together imply such refined conjectures for base fields that are abelian extensions of $\Q$. As an application of our main result, we show that the Kida formula obtained by Matsuno in  \cite{MATSUNO200080} for $p$-ordinary elliptic curves and by Pollack--Weston in \cite{PWKida} for $p$-supersingular modular forms follows from our Kida formula for Mazur--Tate elements. On another front, we extend Matsuno's method to show a Kida formula holds for the $p$-adic $L$-function of elliptic curves with additive reduction defined by Delbourgo in \cite{del-compositio}. The algebraic analogue of this was proved in \cite{RSadditivekida}. 

\par We expect there to be applications of our Kida formula to obtain formulae for the growth of Tate--Shafarevich groups in $p$-extensions of abelian number fields.
\subsection{Organization} Including the introduction, the article consists of four sections. In section \ref{s 2} we discuss preliminary notions and basic properties of Mazur--Tate elements, their relation to L-values and their Iwasawa invariants. In section \ref{s 3}, we prove the main results of this article, establishing an analogue of Kida's formula for elliptic curves over $\Q$ and more generally for Hecke eigencuspforms. Section \ref{s 4} is devoted to various applications of Kida's formula for Mazur--Tate elements.
\subsection{Acknowledgement} 
This research was initiated at the International Centre for Theoretical Sciences during the program "Automorphic Forms and the Bloch–Kato Conjecture" (ICTS/afbk2025/05) held from 19th to 30th May 2025 in Bangalore, India. The authors thank the institute for the stimulating environment conducive to research.
The authors would like to thank Robert Pollack for interesting conversations related to this paper. We thank the anonymous referees for their referee reports, which led to significant improvement of the manuscript from its earlier version.
\subsection*{Conflict of interest} The authors have no conflict of interest to report.
\subsection*{Data Availability} There is no data associated to the results of this manuscript.
\section{Mazur--Tate elements}\label{s 2}
\subsection{Basic notions}
\par Let us introduce Mazur--Tate elements associated to Hecke eigencuspforms, and their basic properties. Let $k\geq 2$ be an integer and $g:=k-2$. Given a commutative ring $R$ with $1$, set $V_g(R):=\op{Sym}^g(R^2)$, which we regard as the space of homogeneous polynomials of degree $g$ in two variables $X,Y$ with coefficients in $R$. For $\gamma=\smallmtx{a}{b}{c}{d}\in \op{GL}_2(R)$, set $\gamma^*:=\smallmtx{d}{-b}{-c}{a}$. The $R$-module $V_g(R)$ is endowed with a natural right action of $\op{GL}_2(R)$ defined by
\[(P\mid \gamma)(X,Y)\;=\;P\bigl((X,Y)\gamma^*\bigr)\;=\;P(dX-cY,-bX+aY),\]
\noindent for $P\in V_g(R)$ and $\gamma=\smallmtx{a}{b}{c}{d}\in \op{GL}_2(R)$. Note that when $k=2$, $V_0(R)=R$ with trivial action of $\op{GL}_2(R)$. Let $\Gamma\subseteq \op{SL}_2(\Z)$ be a congruence subgroup.
\begin{definition}[Modular symbols]
    By a modular symbol of weight $k$ over the ring $R$, we mean an element in the compactly supported cohomology group $H^1_c(\Gamma, V_g(R))$, where $g:=k-2$ and $V_g(R)$ has the discrete topology.
\end{definition}
There is a canonical Hecke-equivariant isomorphism:
\begin{equation}\label{basic mod symbol iso}
H_c^1(\Gamma,V_g(R)) \;\;\cong\;\; \Hom_\Gamma\!\left(\op{Div}^0(\mathbb{P}^1(\Q)),V_g(R)\right),
\end{equation}
(cf.~\cite[Proposition 4.2]{AshStevens}) where $\op{Div}^0(\mathbb{P}^1(\Q))$ denotes the group of degree-zero divisors on the cusps $\mathbb{P}^1(\Q)$. The right-hand side consists of additive maps $\varphi:\op{Div}^0(\mathbb{P}^1(\Q))\to V_g(R)$ satisfying the equivariance condition \begin{equation}\label{equivariance condition}
    \varphi(\gamma D)\mid \gamma = \varphi(D)
\end{equation}\noindent for all $\gamma\in \Gamma$. In what follows, we identify these spaces via the above isomorphism \eqref{basic mod symbol iso} and refer to them interchangeably as spaces of \emph{modular symbols}.

\par Associated with a modular symbol is a sequence of group-ring elements known as the \emph{Mazur--Tate elements}.  
Fix an integer $M$ coprime to $p$, and define
\[
\Delta_M := \operatorname{Gal}(\mathbb{Q}(\mu_M)/\mathbb{Q}) 
\quad \text{and}\quad
\Delta := \operatorname{Gal}(\mathbb{Q}(\mu_p)/\mathbb{Q}).
\]
For each integer $n \ge 0$, set $\mathscr{G}_n^M := \operatorname{Gal}\!\bigl(\mathbb{Q}(\mu_{p^n M})/\mathbb{Q}\bigr)$.
This group decomposes as
\[
\mathscr{G}_{n+1}^M \;\cong\; G_n \times \Delta \times \Delta_M,
\]
where $G_n := \operatorname{Gal}\!\bigl(\mathbb{Q}(\mu_{p^{n+1}})/\mathbb{Q}(\mu_p)\bigr)$. For brevity, write $\mathscr{G}_n := \mathscr{G}_n^1$ when $M=1$. Let 
\[\phi_n: \mathscr{G}_{n+1}\twoheadrightarrow G_n\]
denote the projection onto the first component of the above decomposition. The cyclotomic character induces canonical isomorphisms
\[
\chi_n^M : \mathscr{G}_n^M \xrightarrow{\;\sim\;} (\mathbb{Z}/p^n M\mathbb{Z})^\times,
\qquad
\text{and set } \chi_n := \chi_n^1.
\]
Let $\omega: \mathscr{G}_n \rightarrow  \mathbb{Z}_p^\times$ denote the Teichmüller character.
\noindent Set $\mathcal{G}:=\op{Gal}(\Q(\mu_{p^\infty M})/\Q)$ and $\Gamma:=\op{Gal}(\Q(\mu_{p^\infty})/\Q(\mu_p))$. Observe that there is a natural decomposition 
\[\mathcal{G}=\Gamma\times \Delta \times \Delta_M.\] Fix a discrete valuation ring $\cO$ containing $\Z_p$ and let $\mathfrak{G}$ be a profinite group. Then the Iwasawa algebra of $\mathfrak{G}$ over $\cO$ is defined as the completed group algebra 
\[\Lambda(\mathfrak{G}):=\varprojlim_{N} \cO[\mathfrak{G}/N],\] where $N$ ranges over all finite index normal subgroups of $\mathfrak{G}$. For ease of notation, we shall set $\Lambda:=\Lambda(\Gamma)$.
\par We fix a topological generator $\gamma\in \Gamma$ and let $\gamma_n$ be the image of $\gamma$ in $G_n$ with respect to the natural quotient $\Gamma\twoheadrightarrow G_n$. Setting $T:=(\gamma-1)$, one identifies $\Lambda$ with the formal power series ring $\cO\llbracket T\rrbracket$. Set $\Lambda_n:=\cO[G_n]$ and note that 
\[\Lambda_n\simeq \frac{\cO[T]}{\left((1+T)^{p^n}-1\right)}.
\]
For $a\in (\Z/p^{n+1}M\Z)^\times$, define $t_n(a)\in (\Z/p^n\Z,+)$ such that $\chi_n^M(\gamma_n)^{t_n(a)}\equiv \langle a \rangle \pmod{p^{n+1}}$, where $\langle a\rangle$ is the projection of $a$ onto the second factor in the decomposition:
\[(\Z/p^{n+1}M\Z)^\times \simeq (\Z/pM\Z)^\times\times (\Z/p^n\Z,+).\] Note that since the elements $\gamma_n$ all arise from $\gamma$, it follows that $t_m(a)\equiv t_n(a)\pmod{p^m}$ whenever $m\leq n$. Throughout, we fix an embedding 
\(\iota:\overline{\Q}\hookrightarrow \overline{\Q}_p\).
Let \[\psi:\mathscr{G}^M_{n+1}\rightarrow \bar{\Q}^\times\] be a character whose image lies in $\cO^\times$. We view $\psi$ as a Dirichlet character of modulus $p^{n+1}M$ under our identification $\mathscr{G}^M_{n+1}\cong (\Z/p^{n+1}M\Z)^\times$. Via the natural projection 
\[
\mathscr{G}_n^M \twoheadrightarrow \mathscr{G}_n,
\]
we regard $\omega$ and $\phi_n$ as group homomorphisms on $\mathscr{G}_n^M$. We get a group homomorphism
\[
\psi\phi_n:\mathscr{G}_{n+1}^M\longrightarrow \cO[G_n]^\times
\quad \text{defined by}\quad \sigma\longmapsto \psi(\sigma)\phi_n(\sigma).
\]
Extending \(\cO\)-linearly, we obtain a homomorphism
\[
\psi\phi_n:\cO[\mathscr{G}_{n+1}^{M}]
\longrightarrow
\cO[G_n]
=
\Lambda_n
\simeq
\frac{\cO[T]}{\bigl((1+T)^{p^n}-1\bigr)},
\]
provided that \(\cO\) contains the values of \(\psi\). We use this map to define Mazur--Tate elements twisted by $\psi$ as polynomials in $T$. Let $\{r\}$ denote the divisor class of the cusp $r\in \mathbb{P}^1(\mathbb{Q})$, so that $\{\infty\}-\{r\}$ is a degree-zero divisor. Given a Dirichlet character $\psi$, let $\Z_p[\psi]$ be the extension of $\Z_p$ generated by the values of $\psi$.
\begin{definition}[Mazur--Tate elements]\label{defn: mt elements}
Let $p>2$ be prime and let $M\geq1$ be an integer such that $p\nmid M$. The classical Mazur--Tate element of level $n$ associated with $\varphi \in H^1_c(\Gamma, V_g(R))$ is defined as
\[\Theta_n^M(\varphi) := \sum_{a \in (\Z/p^{n+1}M\Z)^\times} \varphi\left(\{\infty\}-\left\{\frac{a}{p^{n+1}M}\right\}\right)|_{(X,Y)=(0,1)}\cdot \sigma_a \in \cO[\mathscr{G}_{n+1}^M],\]
where $\chi_{n+1}^M(\sigma_a)=a$ (recall $\chi_{n+1}^M$ is the cyclotomic character used to identify $\mathscr{G}_{n+1}^M \cong (\Z/p^{n+1}M\Z)^\times$).
\par Let $\psi$ be a Dirichlet character of conductor $M'p^{m}$ for $m\leq n$ and $M'\mid M$. Assume that $\cO$ contains $\Z_p[\psi]$. Define the Mazur--Tate element of level $n$ associated with $\varphi \in H^1_c(\Gamma, V_g(R))$ twisted by $\psi$ as
\begin{align*}
    \Theta_n^M(\varphi,\psi,T) :=& \psi\phi_n(\Theta_n^M(\varphi))\\
    =&\sum_{a \in (\Z/p^{n+1}M\Z)^\times} \varphi\left(\{\infty\}-\left\{\frac{a}{p^{n+1}M}\right\}\right)|_{(X,Y)=(0,1)}\cdot \psi(a)\cdot(1+T)^{t_n(a)}\in \Lambda_n.
\end{align*}
   When $M=1$ and $\psi$ is the trivial character, we simply denote the Mazur--Tate element by $\Theta_n(\varphi,T)$.
\end{definition}

\par It is worth noting that Mazur--Tate elements are often defined as elements of the group ring $\Lambda_n$ for characters $\psi$ of modulus prime to $p$ (cf., for instance, \cite{pollack05} or \cite[p.~357]{PW}). On the other hand, we allow twists by characters $\psi$ having conductor divisible by powers of $p$.
Indeed, we shall show that this modified Mazur--Tate element satisfies interpolation properties similar to the usual Mazur--Tate elements (see Proposition \ref{prop: interpolation property}).
\par Let $K$ be the field of fractions of $\cO$. Associate Iwasawa invariants $\mu(G)$ and $\lambda(G)$ to a power series $G=\sum_j b_j T^j\in \Lambda$ as follows. By the Weierstrass preparation theorem, $G(T)=\pi^{m} P(T)u(T)$ where $P(T)$ is a \emph{distinguished polynomial} and $u(T)$ is a unit in $\Lambda$. This means that $P(T)$ is a monic polynomial whose non-leading coefficients are divisible by $\pi$ and $u(0)\in \cO^\times$. Given such a decomposition, 
\[\mu(G):=\frac{m}{e_{K/\Q}}\quad \text{and}\quad \lambda(G):=\op{deg}P(T),\] where $e_{K/\Q}$ is the ramification index of $K$ over $\Q$. On the other hand, the $\mu$ and $\lambda$-invariants of $F=\sum_{i=0}^{p^n-1}a_iT^i \in \Lambda_n$ are defined as 
\begin{equation}\label{finite mu and lambda}\begin{split}
    \mu(F) &= \underset{i}{\min}\{\ord_{p}(a_i)\},\\
    \lambda(F) &= \min\{ i : \ord_{p}(a_i) = \mu(F)\}
\end{split}
\end{equation}
where $\ord_{p}$ is the $p$-adic valuation such that $\ord_p(p)=1$. If $F=0$, set $\mu(F),\lambda(F):=\infty$. We note that $\op{ord}_p(\cdot)$ has values in $\Q$ when $\cO$ is a ramified extension of $\Z_p$, and hence so does $\mu(F)$. Note however that $\lambda(F)$ is always an integer and is independent of the choice of $\cO$. 
\begin{remark}One may define $\mu$ and $\lambda$-invariants for an element in $\Lambda_n$ intrinsically, purely in terms of the augmentation ideal (see \cite[\S~3.1]{PW}). Thus, the Iwasawa invariants are independent of the choice of the generator $\gamma_n$.
\end{remark}
\subsection{Modular symbols attached to eigenforms}
\par Next, we discuss the definition and properties of symbols $\varphi$ which arise from a normalized Hecke eigencuspform $f=\sum a_{n}(f)q^n\in S_k(\Gamma_0(N),\mathbb{C})$ for some integers $k\geq 2$ and $N\geq 1$. Denote by $K_f$ the Hecke field of $f$, i.e., the number field generated by the Fourier coefficients $\{a_n(f)\mid n\geq 1\}$. Let $\xi_f$ be the modular symbol defined by:
$$
\xi_f(\{r\}-\{s\})\;=\; 2\pi i\int_s^r f(z)(zX+Y)^{k-2}\,dz,\quad \text{where}\quad r,s\in \mathbb{P}^1(\Q).
$$
\noindent The action of the Hecke operators on $H^1_c(\Gamma_0(N),V_{k-2}(\mathbb{C}))$ is given by
    \[T_\ell:= \epsilon_N(\ell)\begin{psmallmatrix}
        \ell & 0 \\ 0 & 1
    \end{psmallmatrix} + \sum_{j=0}^{\ell-1} \smallmtx{1}{j}{0}{\ell},\]
    where $\epsilon_N$ denotes the character on $\Z$ defined as
   \[\epsilon_N(a)=\begin{cases}
       1 \text{ if } (a,N)=1\\
       0 \text{ if } (a,N)\neq 1.
   \end{cases}\]
It is a well known fact that the symbol $\xi_f$ is a Hecke eigensymbol with the same eigenvalues as $f$, i.e., $T_\ell \varphi_f= \varphi_{T_\ell f}$ (see \cite[l.5 of \S~2.2]{PW}).
\par The involution $\iota=\begin{psmallmatrix}
        -1 & 0 \\ 0 & 1
    \end{psmallmatrix}$ acts on the space of modular symbols, and $\xi_f$ decomposes uniquely as $\xi_f^++\xi_f^-$ with $\xi_f^\pm$ in the $\pm 1$-eigenspaces. By a theorem of Shimura \cite{shimura} (also see \cite[Proposition 5.11]{PacsolPopa}), there exist complex numbers $\Omega_f^\pm$ (the $\pm$-periods of $f$) such that $\xi_f^\pm$ takes values in $V_{k-2}(K_f)\Omega_f^\pm$, where $K_f$ is the Hecke field of $f$. After fixing an embedding $\overline{\Q}\hookrightarrow \overline{\Q}_p$, normalize $\xi_f^\pm$ by setting $\varphi_f^\pm:=\xi_f^\pm/\Omega_f^\pm$, which then takes values in $V_{k-2}(\overline{\Q}_p)$. Write $\varphi_f:=\varphi_f^++\varphi_f^-$, which depends on the chosen periods. Given $r \in \mathbb{P}^1(\Q)$, set 
\[\varphi_f(r):=\varphi_f(\{\infty\}-\{r\})\] for ease of notation.
\par Let $\mathcal{O}_f$ denote the ring of integers of the $p$-adic completion of $K_f$. It is convenient to choose $\Omega_f^\pm$ so that $\varphi_f^\pm$ has coefficients integral at $p$. More precisely, $\Omega_f^\pm$ are called \emph{cohomological periods} for $f$ if $\varphi_f^\pm$ takes values in $V_{k-2}(\mathcal{O}_f)$, and there exists $r, s\in \mathbb{P}^1(\Q)$ such that $\varphi_f^\pm(\{r\}-\{s\})$ has at least one coefficient which is a $p$-adic unit. Such periods always exist and are unique up to multiplication by a $p$-adic unit in $K_f$. Fixing $\psi$, we choose a valuation ring $\cO$ containing $\cO_{f}(\psi)$. In practice, $\cO$ will also be unramified over $\Z_p$. With this normalization, one defines the Mazur–Tate elements attached to $f$ as follows:
\[\Theta_n^M(f,\psi,T)\;=\;\Theta_n^M(\varphi_f,\psi,T) \quad \text{and}\quad \Theta_n^M(f)\;=\;\Theta_n^M(\varphi_f).\] 
\par We recall some fundamental properties of modular symbols. Since $\varphi_f$ is a Hecke eigensymbol, 
\begin{equation}\label{eqn: Hecke relation for modsymbs}
a_\ell\varphi_f(r)=\epsilon_N(\ell)\ell^{k-2}\varphi_f(\ell r)+\sum_{j=0}^{\ell-1}\varphi_f((r+j)/\ell),
\end{equation}
\noindent for all $r \in \mathbb{P}^1(\Q)$.

For integers $n > m\geq 0$, there are natural projection maps
\[
\pi_{n,m}^{M} : \mathcal{O}[\mathscr{G}_{n}^{M}] \longrightarrow \mathcal{O}[\mathscr{G}_{m}^{M}],
\]
arising from the canonical surjection 
\[
\mathscr{G}_{n}^{M} = \operatorname{Gal}\!\bigl(\mathbb{Q}(\mu_{p^{n}M})/\mathbb{Q}\bigr)
\twoheadrightarrow 
\operatorname{Gal}\!\bigl(\mathbb{Q}(\mu_{p^{m}M})/\mathbb{Q}\bigr) = \mathscr{G}_{m}^{M}.
\]
Similarly, for integers $M \mid M'$, we define projection maps
\[
\pi_{n}^{M',M} : \mathcal{O}[\mathscr{G}_{n}^{M'}] \longrightarrow \mathcal{O}[\mathscr{G}_{n}^{M}],
\]
induced by the natural quotient
\[
\mathscr{G}_{n}^{M'} = \operatorname{Gal}\!\bigl(\mathbb{Q}(\mu_{p^{n}M'})/\mathbb{Q}\bigr)
\twoheadrightarrow 
\operatorname{Gal}\!\bigl(\mathbb{Q}(\mu_{p^{n}M})/\mathbb{Q}\bigr) = \mathscr{G}_{n}^{M}.
\]
Under the cyclotomic identifications
\[
\chi_{n}^{M} : \mathscr{G}_{n}^{M} \xrightarrow{\;\sim\;} (\mathbb{Z}/p^{n}M\mathbb{Z})^{\times},
\]
the maps $\pi_{n,m}^{M}$ and $\pi_{n}^{M',M}$ correspond, respectively, to the natural reduction maps
\[
(\mathbb{Z}/p^{n}M\mathbb{Z})^{\times} \longrightarrow (\mathbb{Z}/p^{m}M\mathbb{Z})^{\times}
\qquad\text{and}\qquad
(\mathbb{Z}/p^{n}M'\mathbb{Z})^{\times} \longrightarrow (\mathbb{Z}/p^{n}M\mathbb{Z})^{\times}.
\]

We also define the \emph{trace map}
\[
\nu_{n-1,n}^{M} : \mathcal{O}[\mathscr{G}_{n-1}^{M}] \longrightarrow \mathcal{O}[\mathscr{G}_{n}^{M}]
\]
by
\[
\nu_{n-1,n}^{M}(\tau) \;=\; \sum_{\pi(\sigma) = \tau} \sigma, \qquad \tau \in \mathscr{G}_{n-1}^{M},
\]
where $\pi : \mathscr{G}_{n}^{M} \twoheadrightarrow \mathscr{G}_{n-1}^{M}$ denotes the natural projection.
Under the cyclotomic identification
\[
\mathscr{G}_{n}^{M} \;\xrightarrow{\;\chi_{n}^{M}\;}\; (\mathbb{Z}/p^{n}M\mathbb{Z})^{\times},
\]
the map $\nu_{n-1,n}^{M}$ corresponds to the operation $h(T) \mapsto \tilde{h}(T)\,\omega_{n}(T)$, where $\widetilde{h}(T)$ is lift of $h(T)$ with respect to the map $\pi_{n,n-1}^{M}$ and $\omega_{n}(T):=\sum_{j=0}^{p-1} (1+T)^{j p^{n-1}}$. The following result shows that Mazur--Tate elements satisfy certain compatibility relations with respect to  projection maps defined above. In fact, these relations are derived from \eqref{eqn: Hecke relation for modsymbs}.
\begin{lemma}[Compatibility with $\pi^{M\ell,M}_n$]\label{lem: pi n compatibility}
   Let $\ell$ be a prime such that $\ell\nmid Mp$ and $\psi$ be a character of conductor $p^mM$ for $m\leq n$. Then, we have that
   \[\pi^{M\ell,M}_n(\Theta^{M\ell}_n(f))=(a_\ell(f)-\sigma_\ell-\epsilon_N(\ell)\ell^{k-2}\sigma_\ell)\Theta^{M}_n({f}).\]
   Thus, we have
    \[\Theta^{M\ell}_n(f, \psi,T)= h_\ell(\psi,T)\Theta^{M}_n({f},\psi,T),\] with 
    \[h_\ell(\psi, T):=a_\ell- \psi(\ell)(1+T)^{
t_n(\ell)}- \epsilon_N(\ell)\ell^{k-2}\psi(\ell)^{-1}(1+T)^{-t_n(\ell)}.\]
\end{lemma}
\begin{proof}
    From the Hecke equivariance relation \eqref{eqn: Hecke relation for modsymbs}, it follows that
    \begin{align*}
        a_\ell \varphi_{f}\left(\frac{a\ell}{p^{n+1}M}\right) 
       &=\epsilon_N(\ell)\ell^{k-2}\varphi_f\left(\frac{a\ell^2}{p^{n+1}M}\right)+ \varphi_f\left(\frac{a}{p^{n+1}M}\right)+\sum_{j=1}^{\ell-1} \varphi_f\left(\frac{a\ell+jp^{n+1}M}{p^{n+1}M\ell}\right).
    \end{align*}
Note that 
\[\psi(a\ell+jp^{n+1}M)=\psi(a\ell) \quad \text{and}\quad t_n(a\ell+jp^{n+1}M)=t_n(a\ell).\]Thus we deduce that
\begin{align*}
   \pi^{M\ell,M}_{n}(\Theta^{M\ell}_n (f)) &= \sum_{\substack{a=0 \\ (a, pM)=1}}^{(p^{n+1}M-1)} \left[\sum_{j=1}^{\ell-1} 
\varphi_f\left(\frac{a\ell+jp^{n+1}M}{p^{n+1}M\ell}\right)|_{(X,Y)=(0,1)}\right]\ \sigma_{a\ell} \\
&= a_\ell\cdot\sum_{\substack{a=0 \\ (a, pM)=1}}^{(p^{n+1}M-1)} \varphi_f(a\ell/p^{n+1}M)|_{(X,Y)=(0,1)} \cdot \sigma_{a\ell}\\ &  - \epsilon_N(\ell)\ell^{k-2}\sum_{\substack{a=0\\ (a, pM)=1}}^{(p^{n+1}M-1)} \varphi_f\left(\frac{a\ell^2}{p^{n+1}M}\right)|_{(X,Y)=(0,1)}\cdot\sigma_{a\ell} -\sum_{\substack{a=0 \\ (a, pM)=1}}^{(p^{n+1}M-1)} \varphi_f\left(\frac{a}{p^{n+1}M}\right)|_{(X,Y)=(0,1)}\cdot\sigma_{a\ell} \\
&=(a_\ell(f)-\sigma_\ell-\epsilon_N(\ell)\ell^{k-2}\sigma_\ell)\Theta^{M}_n({f})
\end{align*}
\noindent proving the result. Apply $\psi\phi_n$ on both sides, to get 
\[\psi\phi_n(\pi^{M\ell,M}_{n}(\Theta_{n}^{M\ell}(f)))= h_\ell(\psi,T)\Theta_{n}^{M}(f,\psi,T).\]
Finally, note that since $\psi$ has conductor $p^mM$, we have $\psi(a)=\psi(a \pmod{p^mM})$ for $a\in (\Z/p^{n+1}M\ell\Z)$, this implies 
\[\psi\phi_n\left(\pi^{M\ell,M}_{n}(\Theta_{n}^{M\ell}(f))\right)= \psi\phi_n(\Theta_{n}^{M\ell}(f))=\Theta_{n}^{M\ell}(f,\psi,T).\]
\end{proof}

    The polynomial $h_\ell(\psi,T)$ can be understood as an Iwasawa theoretic Euler factor, i.e., an element of $\Lambda_n$ that interpolates the Euler factor at $\ell$. In the proof of Lemma \ref{lem: pi n compatibility}, the key input was that $\varphi_f$ is an eigensymbol for the Hecke operator $T_\ell$ for $\ell\neq p$. Similarly, we can use that $\varphi_f$ is an eigensymbol for the Hecke operator $T_p$ to get the following compatibility relation:

\[\pi^{M}_{n,n-1}(\Theta^{M}_n(f, \psi,T))= a_p(f)\Theta^{M}_n({f},\psi,T)- \epsilon_N(p)p^{k-2}\nu_{n-1,n}(\Theta^M_{n-1}(f,\psi,T)).\]
For a proof, see \cite[\S~4, (4.2)]{MTT}. Now, define the sequence $c_{n-i+1}$ recursively, as follows:
$c_0=0, c_1=1$ and $c_m=a_pp^{-1}c_{m-1}-\epsilon_N(p)p^{k-3}c_{m-2}$. 
Then, one can show by induction, that for integers $n>i\geq 0$, the following holds:
\[\pi^{M}_{n,i}(\Theta^{M}_n(f, \psi,T)))= p^{n-i}c_{n-i+1}\Theta^{M}_i({f},\psi,T)- \epsilon_N(p)p^{k-2} p^{n-i-1}c_{n-i}\nu_{i-1,i}(\Theta^M_{i-1}(f,\psi,T))\]
(see \cite[\S~1]{kuri2002} for further details). It follows that
\begin{equation}\label{i is smaller than n}\Theta^{M}_n(f, \psi,\zeta_{p^i}-1)=\pi^{M}_{n,i}\left(\Theta^{M}_n(f, \psi,\zeta_{p^i}-1)\right)=p^{n-i}c_{n-i+1}\Theta^{M}_i({f},\psi,\zeta_{p^i}-1)\end{equation}
where $\zeta_{p^i}$ denotes a primitive $p^i$th root of unity. The second equality holds because $\nu_{i-1,i}(\Theta_{i-1}^M(f,\psi,\zeta_{p^i}-1))= 0$ from the definition of the map $\nu_{i-1,i}$. Note that evaluating $\Theta_{n}^M(f,\psi,T)$ at $T=\zeta_{p^i}-1$ corresponds to the value (at $s=1$) of the $L$ function of $f$ twisted by a character of conductor $p^{i+1}$ and order $p^i$ (i.e. a primitive character on $G_i\cong\op{Gal}(\Q_{(i)}/\Q)$) via the interpolation property, which we discuss below in Proposition \ref{prop: interpolation property}.
\subsection{Interpolating $L$-values: Birch's lemma}
We recall some basic computations from \cite[\S8]{MTT} relating Mazur--Tate elements to special values of $L$-functions. 
Let $\chi$ be a Dirichlet character of conductor $m$, and adopt the shorthand $e(z):=e^{2\pi i z}$. The associated Gauss sum is defined by
\[
\tau(n,\chi) := \sum_{a \bmod m} \chi(a)\, e(na/m)
\]
and set $\tau(\chi) := \tau(1,\chi)$. Then $\tau(n,\chi) = \overline{\chi}(n)\, \tau(\chi)$ for all $n \in \mathbb{Z}$.

Let $N\geq 1$ be an integer and let $f=\sum_{n\geq1}a_n(f)e(nz)$ be a Hecke eigencuspform in $S_2(\Gamma_0(N))$. The $L$-function associated with $f$ is computed via the Mellin transform
\begin{equation}\label{mellin trans eqn}
    L(f,s)=\sum_{n\geq 1} a_n(f) n^{-s}=\frac{2\pi^s}{\Gamma(s)}\int_0^\infty f(it)t^{s-1}dt.
\end{equation}

\noindent The $\chi$-twist of $f$ is defined as follows:
\[f_{\chi}(z):=\sum_na_n(f)\chi(n)e(n z).\] 
\begin{lemma}[Birch's lemma]\label{birch lemma}
  Let $\chi$ be a Dirichlet character of conductor $m$. Then
      \[\begin{split} f_{\overline\chi}(z)=& \frac{1}{\tau(\chi)}\sum_{a\op{mod} m}\chi(a) f\left(z+\frac{a}{m}\right).\end{split}\]
\end{lemma}
\begin{proof}
  Since $\chi$ is primitive, we have the relation $\overline{\chi}(n)=\tau(n,\chi)/\tau(\chi)$ for all $n$. Thus, we find that 
    \begin{align*}
        f_{\overline\chi}(z) =& \frac{1}{\tau(\chi)}\sum_{(n,m)=1}a_n\tau(n,\chi)e(nz)\\
        = & \frac{1}{\tau(\chi)}\sum_{a \op{mod}m}\chi(a)\sum_{n=1}^\infty a_n e\left(n\left(z+\frac{a}{m}\right)\right) \\
        =&\frac{1}{\tau(\chi)}\sum_{a \op{mod}m}\chi(a) f\left(z+\frac{a}{m}\right).\end{align*}
\end{proof}
\noindent Let $\chi_n: G_n\rightarrow \mu_{p^n}$ denote the character $\gamma_n \mapsto \zeta_{p^n}$. Since we have identified $G_n$ with the $p$-Sylow subgroup of $(\Z/p^{n+1}\Z)^\times$, $\chi_n$ gives rise to a Dirichlet character of conductor $p^{n+1}$. For a Dirichlet character \(\eta\), set\[
L(f,\eta,s)
:=
\sum_{r \geq 1} a_r(f)\eta(r)r^{-s}.
\]With this notation, Lemma \ref{birch lemma} gives the following interpolation formula.
\begin{proposition}\label{prop: interpolation property}
    Let $M\geq 1$ be an integer such that $(p,M)=1$. Then 
    \begin{itemize}
        \item[(a)] If $\psi$ is a Dirichlet character of conductor $Mp^m$ for $m\leq n$, then 
        \[\Theta^M_{n}(f, \psi, \zeta_{p^n}-1)= {\tau(\psi\chi_n)}\frac{L(f,\overline{\psi\chi_n},1)}{\Omega_f^{\psi(-1)}},\]
       \item[(b)]For $i<n$, we have
        \[\Theta^M_{n}(f, \psi, \zeta_{p^i}-1)= {p^{n-i}c_{n-i+1}}{\tau(\chi_i\psi)}\frac{L(f,\overline{\psi\chi_i},1)}{\Omega_f^{\psi(-1)}}\]
  where the $c_m$ are defined recursively by $c_0=0, c_1=1$ and $c_m=a_pp^{-1}c_{m-1}-\epsilon_N(p)p^{-1}c_{m-2}$ for $m\geq2$.
\end{itemize}
\end{proposition}
 \begin{remark}\label{rem:constant-term-MT}
Taking \(i=0\) in Proposition~\ref{prop: interpolation property} gives
\[
\Theta^M_n(f,\psi,0)
=
p^n c_{n+1}\tau(\psi)
\frac{L(f,\psi,1)}{\Omega_f^{\psi(-1)}}
\]
for a character \(\psi\) of conductor \(M\) coprime to \(p\). Thus, the constant term of
\(\Theta^M_n(f,\psi,T)\) can be used to recover \(L(f,\psi,1)\), provided
\(c_{n+1}\neq 0\). Also, note that if
\[
\lambda\bigl(\Theta^M_n(f,\psi,T)\bigr)=0,
\]
then \(L(f,\psi,1)\neq 0\).
\end{remark}
\begin{proof}[Proof of Proposition \ref{prop: interpolation property}]
\par It suffices first to prove part (a). Indeed, once part (a) is known, part (b) follows from \eqref{i is smaller than n}. Applying the Mellin transform as in \eqref{mellin trans eqn} to the statement of Lemma \ref{birch lemma}, we find that:
\begin{align*}\label{L(f, chi, 1) eqn}{\tau(\psi\chi_n)}L(f,\overline{\psi\chi_n},1) =&\sum_{a \op{mod}{p^{n+1}M}} \psi\chi_n(a) \int_{\frac{a}{p^{n+1}M}}^\infty f(z)dz\\ = &\sum_{a \op{mod}{p^{n+1}M}} \psi\chi_n(a) \xi_{f}\left(\{\infty\}-\left\{\frac{a}{p^{n+1}M}\right\}\right)|_{(X,Y)=(0,1)}\\
= &\sum_{a \op{mod}{p^{n+1}M}} \psi(a)\xi_{f}\left(\{\infty\}-\left\{\frac{a}{p^{n+1}M}\right\}\right)|_{(X,Y)=(0,1)}(\zeta_{p^n})^{t_n(a)}.\end{align*}
Here we have implicitly used the fact that $\psi\chi_n$ has conductor $p^{n+1}M$. 
\par Also note that $\xi_f^\pm(\{\infty\}-\{-r\})=\pm\xi_f^\pm(\{\infty\}-\{r\})$ since $\xi_f^\pm$ is an eigenvector for the involution $\iota$ with sign $\pm1$. It then follows that 
\begin{align*}
   &\sum_{a \op{mod}{p^{n+1}M}} \psi\chi_n(a) \xi_{f}^{-\psi(-1)}\left(\{\infty\}-\left\{\frac{a}{p^{n+1}M}\right\}\right)|_{(X,Y)=(0,1)}\\=&\sum_{a \op{mod}{p^{n+1}M}} \psi\chi_n(-a) \xi_{f}^{-\psi(-1)}\left(\{\infty\}-\left\{-\frac{a}{p^{n+1}M}\right\}\right)|_{(X,Y)=(0,1)}\\
   =& -\psi(-1)^2\sum_{a \op{mod}{p^{n+1}M}} \psi\chi_n(a) \xi_{f}^{-\psi(-1)}\left(\{\infty\}-\left\{\frac{a}{p^{n+1}M}\right\}\right)|_{(X,Y)=(0,1)}\\
   =& -\sum_{a \op{mod}{p^{n+1}M}} \psi\chi_n(a) \xi_{f}^{-\psi(-1)}\left(\{\infty\}-\left\{\frac{a}{p^{n+1}M}\right\}\right)|_{(X,Y)=(0,1)}. 
\end{align*}
 Thus this sum equals $0$ and consequently,\[\sum_{a \op{mod}{p^{n+1}M}} \psi\chi_n(a) \xi_{f}\left(\{\infty\}-\left\{\frac{a}{p^{n+1}M}\right\}\right)|_{(X,Y)=(0,1)}= \sum_{a \op{mod}{p^{n+1}M}} \psi\chi_n(a) \xi_{f}^{\psi(-1)}\left(\{\infty\}-\left\{\frac{a}{p^{n+1}M}\right\}\right)|_{(X,Y)=(0,1)}\]
\noindent from which part (a) follows.
\end{proof}

\subsection{Mazur--Tate elements over abelian extensions of $\Q$}
For a character $\kappa$, let $M_{\kappa}$ denote the non-$p$ part of its conductor, i.e., the conductor of $\kappa$ is of the form $p^{n'}M_\kappa$ for some $n'\geq0$ with $(p,M_\kappa)=1$. For a finite abelian extension $K/\Q$, let $n_K$ denote the integer such that $K\cap \Q_\infty=\Q_{(n_K)}$.
Then, from the ramification of primes lying above $p$ in $K$, it can be seen that $K\subset \Q(\mu_{p^{n_K+1}M})$ with $(p,M)=1$. Thus, a character on $\op{Gal}(K/\Q)$, when viewed as a Dirichlet character, has conductor having $p$-adic valuation at most $n_K+1$.
\begin{lemma}\label{lemma about g and h}
Let $\zeta$ be a primitive $p^m$-th root of unity and define
\[
g(T) := \prod_{i=0}^{p^m-1} f(\zeta^i(1+T)-1)
\]
for some $f(T) \in \mathbb{Z}_p[T]$. Then there exists a unique polynomial $h(T) \in \mathbb{Z}_p[T]$ such that
\[
g(T) = h\bigl((1+T)^{p^m} - 1\bigr).
\]
\end{lemma}

\begin{proof}
The proof follows an argument from \cite[Lemma~13.39]{washington2012introduction}. First we prove the existence of $h$. Note that $\deg g = p^m \cdot \deg f$. Moreover,
\begin{equation}\label{eq: g symmetry}
    g\bigl(\zeta^i(1+T)-1\bigr) = g(T)
\end{equation}
for all $0 \leq i \leq p^m-1$. In particular, substituting $T = \zeta^j - 1$ shows that $g(\zeta^j - 1) = g(0)$ for all $0 \leq j \leq p^m - 1$.

It follows that
\[
g(T) = g(0) + \bigl((1+T)^{p^m} - 1\bigr) \cdot g_1(T)
\]
for some $g_1(T) \in \mathbb{Z}_p[T]$ with $\deg g_1 = \deg g - p^m$. By \eqref{eq: g symmetry}, the polynomial $g_1$ also satisfies
\[
g_1\bigl(\zeta^i(1+T)-1\bigr) = g_1(T)
\]
for all $0 \leq i \leq p^m - 1$. Repeating the above argument for $g_1$, we write
\[
g_1(T) = g_1(0) + \bigl((1+T)^{p^m} - 1\bigr) \cdot g_2(T)
\]
with $g_2(T) \in \mathbb{Z}_p[T]$ and $\deg g_2 = \deg g - 2p^m$. Continuing inductively, we obtain
\[
g(T) = \sum_{i=0}^{\deg f} a_i \bigl((1+T)^{p^m} - 1\bigr)^i
\]
for some $a_0, \dots, a_{\deg f} \in \mathbb{Z}_p$. Thus we may take
\[
h(T) := \sum_{i=0}^{\deg f} a_i T^i \in \mathbb{Z}_p[T].
\]

To prove uniqueness, suppose $h_1, h_2 \in \mathbb{Z}_p[T]$ both satisfy
\[
h_1\bigl((1+T)^{p^m} - 1\bigr) = h_2\bigl((1+T)^{p^m} - 1\bigr) = g(T).
\]
Then for every $n \geq m$ and any primitive $p^n$-th root of unity $\zeta_{p^n}$, we have
\[
h_1(\zeta_{p^n}^m - 1) = g(\zeta_{p^n} - 1) = h_2(\zeta_{p^n}^m - 1).
\]
Since this holds for infinitely many distinct values $\zeta_{p^n}^m - 1$, it follows that $h_1 - h_2$ has infinitely many roots and is therefore the zero polynomial. Thus $h_1 = h_2$, proving uniqueness.
\end{proof}
\begin{remark}
From the proof above, we see that $\deg h = \deg f$. In particular, if $f \in \Lambda_n$ for some $n \geq 0$, then $h$ also lies in $\Lambda_n$ when viewed as a polynomial in $T$.
\end{remark}

Consider the product 
\[g(T):=\prod_{\psi \in \widehat{\op{Gal}(K/\Q)}} \Theta_{n+n_K}^{M_\psi}(f,\psi,T) \]
and write each character $\psi \in \widehat{\op{Gal}(K/\Q)}$ as $\psi_1\psi_2$ where $\psi_1$ has order prime to $p$ and $\psi_1$ has $p$-power order (note that this implies $\psi_2$ has order dividing $p^{n_K}$, and every character $\psi_2$ corresponds to $\zeta_{p^{n_K}}^i$ for exactly one $i$ as $i$ varies from $0$ to $p^{n_K}-1$). Then we can rewrite this product as 
\[g(T)=\prod_{\psi_1}\prod_{i=0}^{p^{n_K}-1} \Theta_{n+n_K}^{M_\psi}(f,\psi_1,\zeta_{p^{n_K}}^i(1+T)-1). \]
\begin{figure}[h]
    \centering
    \[\begin{tikzcd}
	{K_{(n)}:=K\mathbb{Q}_{(n+n_K)}} \\
	K & {\mathbb{Q}_{(n+n_K)}} \\
	& {\mathbb{Q}_{(n_K)}} \\
	& {\mathbb{Q}}
	\arrow[no head, from=2-1, to=1-1]
	\arrow[no head, from=2-2, to=1-1]
	\arrow[no head, from=3-2, to=2-1]
	\arrow[no head, from=3-2, to=2-2]
	\arrow[no head, from=4-2, to=3-2]
\end{tikzcd}\]
    \caption{$g(T)$ decomposes as a product of Mazur-Tate elements $\Theta_{n+n_K}^{M_\psi}(f,\psi,T)$ for the extension $\Q_{(n+n_K)}/\Q$}
    \label{fig:placeholder}
\end{figure}
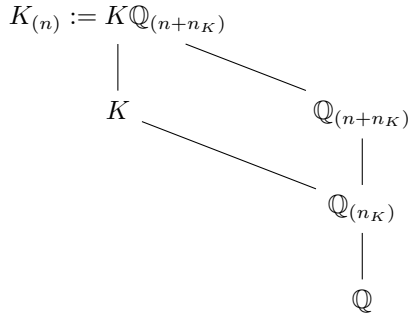

Then, Lemma \ref{lemma about g and h} implies that $g(T)=h((1+T)^{p^{n_k}}-1)$ for some $h \in \Lambda_{n+n_K}$. We define the Mazur--Tate element for $f$ over $K$ via this polynomial $h$ as follows:
\begin{definition}\label{defn: MT elements over K}
Let \(K/\mathbb Q\) be a finite abelian extension, and let \(n \geq 1\). Define
\[
g_{K,n}(T)
:=
\prod_{\psi\in \widehat{\operatorname{Gal}(K/\mathbb Q)}}
\Theta^{M_{\psi}}_{n+n_K}(f,\psi,T)
\in \Lambda_{n+n_K}.
\]
By Lemma~\ref{lemma about g and h}, there exists a unique element
\[
\Theta_n(f/K,T)\in \Lambda_{n+n_K}
\]
such that
\[
\Theta_n\bigl(f/K,(1+T)^{p^{n_K}}-1\bigr)
=
g_{K,n}(T).
\]
We call \(\Theta_n(f/K,T)\) the Mazur--Tate element attached to \(f/K\) at level \(n\).
\end{definition}
\begin{remark}\label{remark: sum decomp for iwasawa invs f/K}
For any $F(T) \in \cO[[T]]$ and some $n\geq1$, the power series $g(T)\coloneqq F((1+T)^{p^n}-1)$ satisfies 
\[\mu(g(T))=\mu(F(T)) \quad\text{and}\quad \lambda(g(T))=p^n\lambda(F(T)).\]
This follows from a direct application of the Weierstrass preparation theorem and noting that $u(T) \in \cO[[T]]^\times$ implies $u((1+T)^{p^n}-1) \in \cO[[T]]^\times$. Thus, we have 
\[\begin{split}&\mu(\Theta_n(f/K,T))=\sum_{\psi\in \widehat{\operatorname{Gal}(K/\mathbb Q)}}
\mu(\Theta^{M_{\psi}}_{n+n_K}(f,\psi,T)),\\
& p^{n_K}\lambda(\Theta_n(f/K,T))=\sum_{\psi\in \widehat{\operatorname{Gal}(K/\mathbb Q)}}
\lambda(\Theta^{M_{\psi}}_{n+n_K}(f,\psi,T)).\end{split}\]
\end{remark}
Let $\rho_f$ be the $\ell$-adic Galois representation attached to $f$. Given a character $\eta:\op{G}_K\rightarrow \bar{\Q}_\ell$, denote by $L(f/K, \eta, s)$ the $L$-function associated to $\rho_{f|\op{G}_K}\otimes \eta$, as in \cite[p.~16]{Tate}.

\begin{proposition}
    For $n\geq1$, let $\chi_{K_{(n)}/K}$ denote a character on $\op{Gal}(K_{(n)}/K)$ of order $p^n$ and let $\zeta_{p^n}$ denote a primitive $p^n$th root of unity. Then, we have 
    \[\Theta_{n}(f/K, \zeta_{p^n}-1)= \prod_{\psi \in \widehat{\op{Gal}(K/\Q)}} \tau(\psi\chi_{n+n_K})\cdot\frac{L(f/K, \chi_{K_{(n)}/K},1)}{(\Omega_f^+)^{r_1}(\Omega_f^-)^{2r_2}},\]
    where $r_1$ (resp. $2r_2$) denotes the number of real (resp. complex) embeddings of $K\hookrightarrow\mathbb{C}$.
\end{proposition}
\begin{proof}
   We have
    \[L(f/K,\chi_{K_{(n)}/K},s)= \prod_{\psi \in \widehat{\op{Gal}(K/\Q)}} L(f,\psi\chi_{n+n_K},s).\]
    To see this, note that the induction to $\op{Gal}(\overline{\Q}/\Q)$ of the character $\chi_{K_{(n)}/K}$ (viewed as a $\op{Gal}(\overline{\Q}/K)$-representation) is $\chi_{{n+n_K}}\otimes \bigoplus_{\psi \in \widehat{\op{Gal}(K/\Q)}} \psi$, from which the factorisation for the $L$-function follows. Then, using the interpolation property of $\Theta_{n+n_K}^{M_\psi}(f,\psi,T)$ from Proposition \ref{prop: interpolation property} we get 
    \[\prod_{\psi \in \widehat{\op{Gal}(K/\Q)}} \Theta_{n+n_K}^{M_\psi}(f,\psi,\zeta_{p^{n+n_K}}-1) = \prod_{\psi \in \widehat{\op{Gal}(K/\Q)}} \tau(\psi\chi_{n+n_K})\cdot\frac{L(f/K, \chi_{K_{(n)}/K},1)}{(\Omega_f^+)^{r_1}(\Omega_f^-)^{2r_2}}\]
    for a primitive $p^{n+n_K}$th root of unity.
    Note that the normalisation of modular symbols on the right hand side is assumed to be using the $\Omega_f^\pm$, so the $L$-value on the right hand side is normalised by $(\Omega_f^+)^{r_1}(\Omega_f^-)^{2r_2}$ where $r_1$ (resp. $2r_2$) denotes the number of real (resp. complex) embeddings of $K\hookrightarrow\mathbb{C}$. This is because each even (resp. odd) character on $\op{Gal}(K/\Q)$ determines a real (resp. complex) embedding into $\mathbb{C}$. 

    Finally, note that we have, by definition,  
    \[\Theta_n(f/K,(\zeta_{p^{n+n_K}})^{p^{n_K}}-1)=  \prod_{\psi \in \widehat{\op{Gal}(K/\Q)}} \Theta_{n+n_K}^{M_\psi}(f,\psi,\zeta_{p^{n+n_K}}-1),\]
    and that any primitive $p^n$th root of unity $\zeta_{p^n}$ can be written as $\zeta_{p^{n+n_K}}^{p^{n_K}}$ for some primitive $p^{{n+n_K}}$th root of unity $\zeta_{p^{n+n_K}}$.
\end{proof}
\section{A Kida-type formula for the $\Theta$ elements}\label{s 3}
\par In this section, we prove our main result. Let $\cO$ be a finite extension of $\Z_p$ and let $\pi$ be a chosen uniformizer. For a cuspidal eigenform $f$ and a character $\chi$ of conductor $p^mM$ with $(p,M)=1$ taking values in $\cO^\times$, for $n\geq m$ define the twisted Iwasawa invariants:
    \[\mu_n^M(\chi)\coloneqq \mu(\Theta_n^M(f,\chi,T))\quad \text{and}\quad\lambda_n^M(\chi)\coloneqq\lambda(\Theta_n^M(f,\chi,T))\]
in accordance with \eqref{finite mu and lambda}.
\begin{proposition}
    Let $\chi, \psi$ be finite order characters on $\mathscr{G}_{n}^M$ taking values in $\cO^\times$. Assume $\chi$ has $p$-power order and $\mu_{n}^{M}(f,\psi)=0$. Then we have 
    \[\mu_{n}^{M}(\chi\psi)=0 \text{ and } \lambda_{n}^{M}(\chi\psi)=\lambda_{n}^{M}(\psi).\]
\end{proposition}
\begin{proof}
    The argument is similar to that in the proof of \cite[Lemma 3.2]{MATSUNO200080}. Given $F,G\in \Lambda_n$ such that $F\equiv G\pmod{(\pi)}$, one has that 
    \[\mu(F)=0\Leftrightarrow \mu(G)=0.\] Further, if the above $\mu$-invariants vanish, then $F\equiv cT^{\lambda(F)}\pmod{\pi \Lambda_n}$ where $c\in (\cO/\pi)^\times$. We deduce that $\lambda(F)=\lambda(G)$. Thus it suffices to show that \begin{equation}\label{congruence theta}\Theta_{n}^{M}(f,\chi\psi,T)\equiv \Theta_{n}^{M}(f,\psi, T) \pmod{\pi\Lambda_{n}}.\end{equation}
    Since $\chi$ has $p$-power order, $\chi(x)\equiv 1 \pmod{\pi}$ for any $x \in (\Z/p^nM\Z)^\times$. This gives us the congruence \eqref{congruence theta} and the result follows.
\end{proof}

We define an integer $g_{\psi,n}(\ell)$ in terms of the decomposition of primes in the finite layers of the cyclotomic $\Zp$-extension $\Q_\infty/\Q$. This will be crucial to our proof of Kida's formula. 
\begin{definition}\label{long definition}
    Given a rational prime $\ell$ and an algebraic extension $F/\Q$, let $g_F(\ell)$ be the number of primes of $F$ lying above $\ell$. If $\ell$ is infinitely decomposed in $F$, set $g_F(\ell):=\infty$. For a character $\psi:\mathscr{G}_{n}^M\rightarrow \cO^\times$ and a Hecke eigen-cuspform $f=\sum_n a_ne(nz)\in S_k(\Gamma_0(N))$, define $g_{\psi, n}(\ell)$ as follows:
    \begin{itemize}
        \item[(i)] If $\psi(\ell)\neq 0$ and $\ell\nmid N$, 
        \[g_{\psi,n}(\ell)=\begin{cases}
            0 & \text{if } a_\ell\not\equiv \psi(\ell)+\ell^{k-2}\psi(\ell)^{-1} \pmod{\pi} \\
            g_{\Q_{(n)}}(\ell) & \text{if } a_\ell\equiv \psi(\ell)+\ell^{k-2}\psi(\ell)^{-1} \pmod{\pi}, \ a_\ell\not\equiv\pm 2\ell^{\frac{k-2}{2}} \pmod{\pi}\\
            2g_{\Q_{(n)}}(\ell) & \text{otherwise.}
        \end{cases}\]
        \item[(ii)] If $\psi(\ell)\neq 0$ and $\ell\mid N$, 
        \[g_{\psi,n}(\ell)=\begin{cases}
            0 & \text{if } a_\ell\not\equiv \psi(\ell) \pmod{\pi} \\
            g_{\Q_{(n)}}(\ell) & \text{if } a_\ell\equiv \psi(\ell) \pmod{\pi}.
        \end{cases}\]
        \item[(iii)] If $\psi(\ell)=0$, then $g_{\psi,n}(\ell)=0$.
    \end{itemize}
\end{definition}
 \noindent We note here that $g_{\psi,n}(\ell)$ is a finite layer variant of $g_{\chi}(\ell)$ considered in \cite{MATSUNO200080}.
\begin{proposition}\label{prop: lambda inv from euler factor}
Let $\psi$ be a Dirichlet character of modulus $p^nM$, and assume that the order of $\psi$ is prime to $p$. Let $\ell$ be a prime such that $(\ell, pM)=1$. Then
\begin{equation}\label{Propn 3.3 formulae}
\mu_n^{M\ell}(\psi)=\mu_n^{M}(\psi)
\qquad\text{and}\qquad
\lambda_n^{M\ell}(\psi)=\lambda_n^{M}(\psi)+g_{\psi,n}(\ell).
\end{equation}
\end{proposition}

\begin{proof}
By Lemma~\ref{lem: pi n compatibility} it suffices to analyze the term
\[
h_\ell(\psi,T)
:=a_\ell-\psi(\ell)(1+T)^{t_n(\ell)}-\epsilon_N(\ell)\ell^{k-2}\psi(\ell)^{-1}(1+T)^{-t_n(\ell)}.
\]
Write $t_n(\ell)=u p^{a}$ with $0\le a\le n-1$ and $u\in(\Z/p^n\Z)^\times$.  Recall that
\[
\chi_n(\sigma_\ell)=\langle\ell\rangle=\chi_n(\gamma_n)^{t_n(\ell)},
\]
so the index of the decomposition group of $\ell$ in $\op{Gal}(\Q_{(n)}/\Q)$ is $p^a$, i.e. $p^{a}=g_{\Q_{(n)}}(\ell)$. It suffices to show that 
\begin{equation}\label{mu lambda for hell}
\mu\left(h_\ell(\psi, T)\right)=0\quad \text{and}\quad \lambda\left(h_\ell(\psi, T)\right)=g_{\psi, n}(\ell).
\end{equation}
We show that not all coefficients of $h_\ell(\psi, T)$ are divisible by $\pi$ and determine the smallest nonzero degree, which is the $\lambda$-invariant of $h_\ell(\psi, T)$. 
\par Working modulo \(\pi\) in the coefficient ring, we use the congruence
\[
(1+T)^{p^{a}} \equiv 1+T^{p^{a}} \pmod{\pi}.
\]
\noindent Expanding $(1+T^{p^{a}})^{u}$ by the binomial theorem yields
\[
(1+T)^{u p^{a}} \equiv 1 + u T^{p^{a}} + \binom{u}{2} T^{2p^{a}}
    \pmod{(\pi,\,T^{3p^{a}})}.
\]
\par First consider the case when $\ell\nmid N$. Here $\epsilon_N(\ell)=1$, so substituting the above expansions into $h_\ell(\psi,T)$ and working modulo $(\pi,T^{3p^{a}})$ gives
\begin{align}\label{eq: expansion-mod-p}
h_\ell(\psi,T)
&\equiv a_\ell -\bigl(\psi(\ell)+\ell^{k-2}\psi(\ell)^{-1}\bigr)\\
&\qquad -\bigl(\ell^{k-2}\psi(\ell)^{-1}-\psi(\ell)\bigr)u\,T^{p^{a}}\\
&\qquad-\tfrac{1}{2}\bigl(\bigl(\psi(\ell)+\ell^{k-2}\psi(\ell)^{-1}\bigr)u^2-\bigl(\psi(\ell)-\ell^{k-2}\psi(\ell)^{-1}\bigr)u)\bigr)T^{2p^{a}}
\pmod{(\pi,T^{3p^{a}})}.\nonumber
\end{align}
Set
\begin{align*}
    A:=&a_\ell-\bigl(\psi(\ell)+\ell^{k-2}\psi(\ell)^{-1}\bigr),\qquad
B:=\ell^{k-2}\psi(\ell)^{-1}-\psi(\ell),\qquad\\
C:=&\bigl(\psi(\ell)+\ell^{k-2}\psi(\ell)^{-1}\bigr)u^2-\bigl(\psi(\ell)-\ell^{k-2}\psi(\ell)^{-1}\bigr)u.
\end{align*}
Now reduce the coefficients modulo $\pi$. The first nonzero term of $h_\ell(\psi,T)$ modulo $\pi$ occurs in one of the following degrees:
\begin{itemize}
    \item If $A\not\equiv 0\pmod{\pi}$ then the constant term is nonzero modulo $\pi$, so the first nonzero term has degree $0$. Thus the $\mu$-invariant and $\lambda$-invariant of $h_\ell(\psi, T)$ are both $0$. On the other hand, if $A\not\equiv 0\pmod{\pi}$, then $g_{\psi,n}(\ell)=0$ according to Definition~\ref{long definition}(i). Thus, the formulae \eqref{mu lambda for hell} are satisfied.
    \item If $A\equiv 0\pmod{\pi}$ but $B\not\equiv 0\pmod{\pi}$ then the term with smallest degree is $-Bu\,T^{p^{a}}$, so \[\lambda\left(h_\ell(\psi, T)\right)=p^{a}=g_{\Q_{(n)}}(\ell).\]  The requirement $A\equiv 0\mod{\pi}$ is equivalent to the condition 
    \[a_\ell\equiv \psi(\ell)+\ell^{k-2} \psi(\ell)^{-1}\pmod{\pi}.\] On the other hand $B\equiv 0\pmod{\pi}$ if and only if $\psi(\ell)\equiv \pm \ell^{\frac{k-2}{2}}\pmod{\pi}$. Therefore in particular, $B\equiv 0\pmod{\pi}$ implies that $a_\ell\equiv \pm 2\ell^{\frac{k-2}{2}}\pmod{\pi}$. Hence if $A\equiv 0\pmod{\pi}$ and $B\equiv 0\pmod{\pi}$ then \[a_\ell\equiv \psi(\ell)+\ell^{k-2}\psi(\ell)^{-1} \pmod{\pi}\] and \[a_\ell\not\equiv\pm 2\ell^{\frac{k-2}{2}} \pmod{\pi}.\]
    Thus, $g_{\psi,n}(\ell)=g_{\Q_{(n)}}(\ell)$ in Definition~\ref{long definition}(i) and it follows that the conditions \eqref{mu lambda for hell} are satisfied.
 \item If $A\equiv 0\equiv B\pmod{\pi}$ then $\psi(\ell)^2\equiv\ell^{k-2}\pmod{\pi}$, and $a_\ell\equiv\pm2\ell^{\frac{k-2}{2}}$ since $p\nmid 2\ell$ we have $C\equiv \pm2\ell^{\frac{k-2}{2}}\not\equiv 0\pmod{\pi}$. The term $-\tfrac{1}{2}CT^{2p^{a}}$
    is therefore the first nonzero term in this case. Thus, the first non-zero term has degree $2p^{a}=2g_{\Q_{(n)}}(\ell)$, matching $g_{\psi,n}(\ell)=2g_{\Q_{(n)}}(\ell)$ in Definition~\ref{long definition}(i).
\end{itemize}
\par Next we assume that \(\ell\mid N\) and thus, $\epsilon_N(\ell)=0$.  Working as above modulo $(p,T^{p^{a}+1})$ we have
\[
h_\ell(\psi,T)\equiv a_\ell-\psi(\ell)(1+T)^{u p^{a}}
\equiv (a_\ell-\psi(\ell))+\psi(\ell)u\,T^{p^{a}}\pmod{(p,T^{p^{a}+1})}.
\]
Reducing coefficients modulo $\pi$ yields:

\begin{itemize}
    \item If $a_\ell\not\equiv\psi(\ell)\pmod{\pi}$ then the constant term is nonzero modulo $\pi$, so the first nonzero term has degree $0$, corresponding to $g_{\psi,n}(\ell)=0$ in Definition~\ref{long definition}(ii).
    \item If $a_\ell\equiv\psi(\ell)\pmod{\pi}$ then the first nonzero term is $\psi(\ell)u\,T^{p^{a}}$, of degree $p^{a}=g_{\Q_{(n)}}(\ell)$, corresponding to $g_{\psi,n}(\ell)=g_{\Q_{(n)}}(\ell)$ in Definition~\ref{long definition}(ii).
\end{itemize}
This completes the proof.
\end{proof}
We have the following proposition for the case of elliptic curves defined over $\Q$. This can be viewed as a finite layer analogue of \cite[Lemma 3.4]{MATSUNO200080}.
\begin{proposition}\label{prop: sum over g_chi}
    Let $K/\Q$ be an abelian extension and assume that $E_{/K}$ satisfies the condition \textbf{(Add)} defined above. Assume that all characters $\chi\in \widehat{\op{Gal}(K/\Q)}$ take values in $\cO^\times$. Let $\ell$ be a prime in $\Q$ and $v$ a prime of $K$ lying above $\ell$. Assume that $p\nmid [K:\Q]$ and $\ell\equiv 1\pmod{p}$. Then, 
    \begin{align*}
        \sum_{\chi \in \widehat{\op{Gal}(K/\Q)}}g_{\chi,n}(\ell)=\begin{cases}
            2g_{K_{(n)}}(\ell) & \text{ if $E/K$ has good reduction at $v$ and $E(K_v)[p]\neq0$,}\\
            g_{K_{(n)}}(\ell) & \text{ if $E/K$ has split multiplicative reduction at $v$,}\\
            0 & \text{ otherwise.}
        \end{cases}
    \end{align*}
\end{proposition}
\begin{proof}
For each $\chi\in\widehat{\op{Gal}(K/\Q)}$, write 
\[
\chi(\ell)\bmod\pi=\zeta_\chi\in\overline{\F}_p^\times.
\]
Since $p\nmid[K:\Q]$, the order of $\zeta_{\chi}$ coincides with that of $\chi(\ell)$. We treat separately the cases $\ell\nmid N$ and $\ell\mid N$. Note that $\ell\neq p$ since it has been assumed in particular that $\ell\equiv 1\pmod{p}$. 
\par First assume that $\ell\nmid N$. In this case $E$ has good reduction at $\ell$, hence at every $v\mid\ell$ of $K$. By the definition of $g_{\chi,n}(\ell)$, the quantity $g_{\chi,n}(\ell)$ is nonzero precisely when the congruence
\begin{equation}\label{al congruence}
a_\ell \equiv \chi(\ell) + \chi(\ell)^{-1} \pmod{\pi}
\end{equation}
holds, and equals $g_{\Q_{(n)}}(\ell)$ except in the case $a_\ell\equiv \pm 2 \pmod{\pi}$, in which case it equals $2g_{\Q_{(n)}}(\ell)$. Since $\ell\equiv1\pmod p$, the above congruence \eqref{al congruence} becomes
\[
a_\ell \equiv \zeta_\chi+\zeta_\chi^{-1}\pmod \pi.
\]
\noindent Let $\alpha\in\overline{\F}_p^\times$ be one of the two roots of the polynomial
\[
x^2-a_\ell x+1 \pmod p.
\]
The other root is $\alpha^{-1}$. The congruence $a_\ell \equiv \zeta_\chi+\zeta_\chi^{-1}\pmod p$ holds if and only if $\zeta_\chi=\alpha$ or $\zeta_\chi=\alpha^{-1}$. Therefore
\[
\sum_{\chi} g_{\chi,n}(\ell)
=
g_{\Q_{(n)}}(\ell)\bigl(\,|\{\chi:\zeta_\chi=\alpha\}|+|\{\chi:\zeta_\chi=\alpha^{-1}\}|\,\bigr),
\]
with the understanding that if $a_\ell\equiv\pm2\pmod p$ then each contributing term equals $2g_{\Q_{(n)}}(\ell)$, but the counting argument is identical.

The quantity $|\{\chi:\zeta_\chi=\alpha\}|$ is the number of characters $\chi$ of $\op{Gal}(K/\Q)$ satisfying $\chi(\ell)\equiv\alpha\pmod\pi$. A standard result (see \cite[Theorem 3.7]{washington2012introduction}) states that if $f$ is the inertial degree of $\ell$ in $K/\Q$, then
\[
|\{\chi:\zeta_\chi=\alpha\}|=\begin{cases}
g_K(\ell) & \text{if }\alpha^f=1,\\[4pt]
0 & \text{otherwise.}
\end{cases}
\]
The same statement holds for $\alpha^{-1}$, since $\alpha^{-1\,f}=1$ if and only if $\alpha^f=1$.

Hence
\[
\sum_{\chi} g_{\chi,n}(\ell)=
\begin{cases}
2g_{\Q_{(n)}}(\ell)g_K(\ell) & \text{if }\alpha^f=1,\\[4pt]
0 &\text{otherwise.}
\end{cases}
\]
Since $p\nmid [K:\Q]$, we have $g_{K_{(n)}}(\ell)=g_{\Q_{(n)}}(\ell)g_K(\ell)$, yielding
\[
\sum_{\chi} g_{\chi,n}(\ell)=
\begin{cases}
2g_{K_{(n)}}(\ell) & \text{if }\alpha^f=1,\\
0 &\text{if }\alpha^f\neq 1.
\end{cases}
\]

Finally, good reduction of $E/K$ at $v$ together with the condition $E(K_v)[p]\neq 0$ is equivalent to the condition $\alpha^f=1$; see \cite[Lemma~3.4]{MATSUNO200080}.
\par Next, assume that $\ell\mid N$. In this setting the definition of $g_{\chi,n}(\ell)$ involves only the congruence
\[
a_\ell\equiv \chi(\ell)\pmod\pi.
\]
Put $\alpha:=a_\ell\bmod p\in\overline{\F}_p$. Then $g_{\chi,n}(\ell)=g_{\Q_{(n)}}(\ell)$ if and only if $\zeta_\chi=\alpha$ and equals $0$ otherwise. By the same character-counting result,
\[
|\{\chi:\zeta_\chi=\alpha\}|=\begin{cases}
g_K(\ell) & \text{if }\alpha^f=1,\\
0 & \text{if }\alpha^f\neq1.
\end{cases}
\]
Thus
\[
\sum_{\chi} g_{\chi,n}(\ell)
=
\begin{cases}
g_{\Q_{(n)}}(\ell)\,g_K(\ell)=g_{K_{(n)}}(\ell) & \text{if }\alpha^f=1,\\[4pt]
0 & \text{otherwise.}
\end{cases}
\]

It remains to interpret the condition $\alpha^f=1$ in terms of the reduction type of $E/K$ at $v$.  
If $E$ has split multiplicative reduction at $\ell$, then $a_\ell=1$, so $\alpha=1$ and therefore $\alpha^f=1$ for all $f$. In this case $E/K$ has split multiplicative reduction at $v$. If $E$ has non-split multiplicative reduction at $\ell$, then $a_\ell=-1$, so $\alpha=-1$; hence $\alpha^f=(-1)^f$, which equals $1$ if and only if $f$ is even. But $E$ becomes split multiplicative over the unramified quadratic extension of $\Q_\ell$, which implies that $f$ is even exactly when $E/K$ has split multiplicative reduction at $v$.

Finally, if $E$ has additive reduction at $\ell$, then $a_\ell\equiv 0\pmod p$, so $\alpha=0$; and since $E$ has additive reduction at all primes of $K$ above $\ell$ by assumption \textbf{(Add)}, we have $g_{\chi,n}(\ell)=0$ for all $\chi$ by the definition. Thus the sum is $0$, which matches the statement of the proposition.

Putting these cases together completes the proof.
\end{proof}

\begin{proposition}[Generalization to cuspidal eigenforms]\label{prop: sum g_chi higher weight}
Let $K/\Q$ be an abelian extension, let $\ell$ be a rational prime, and let $f$ denote the residue degree of $\ell$ in $K/\Q$. Let 
\[
h(q)=\sum_{n\ge1} a_n(h) q^n \in S_k(\Gamma_0(N))
\]
be a cuspidal normalized eigenform of weight $k\ge 2$ with nebentypus~$\beta$. Assume $p\nmid [K:\Q]$ and $\ell\equiv1\pmod p$. Then
\[
\sum_{\chi\in\widehat{\op{Gal}(K/\Q)}} g_{\chi,n}(\ell)
=
\begin{cases}
2\, g_{K_{(n)}}(\ell) & \text{if }\ell\nmid N \text{ and } \alpha^f=1,\\[4pt]
g_{K_{(n)}}(\ell) & \text{if }\ell\mid N \text{ and } \alpha^f=1,\\[4pt]
0 & \text{otherwise,}
\end{cases}
\]
where \(\alpha\in\overline{\F}_p\) is any root of
\[
x^2 - a_\ell(h)\,x + \beta(\ell)\quad\pmod p.
\]
\end{proposition}

\begin{proof}
If \(\ell\nmid N\) the local Euler factor of \(h\) is
\[
1-a_\ell(h)X+\ell^{\,k-1}\beta(\ell)X^2,
\]
which reduces modulo \(p\) (using \(\ell\equiv1\pmod p\)) to
\[
1-a_\ell(h)X+\beta(\ell)X^2.
\]
Let \(\alpha\) be a root of \(x^2-a_\ell(h)x+\beta(\ell)\) in \(\overline{\F}_p\); the other root equals \(\beta(\ell)\alpha^{-1}\). By definition \(g_{\chi,n}(\ell)\) is nonzero exactly when
\[
a_\ell(h)\equiv \chi(\ell)+\ell^{k-1}\beta(\ell)\chi(\ell)^{-1}\pmod p,
\]
which (since \(\ell\equiv1\pmod p\)) is equivalent to
\[
a_\ell(h)\equiv \zeta_\chi+\beta(\ell)\zeta_\chi^{-1}\pmod p,
\]
where \(\zeta_\chi:=\chi(\ell)\bmod\pi\). Thus \(g_{\chi,n}(\ell)\neq0\) precisely when \(\zeta_\chi\in\{\alpha,\beta(\ell)\alpha^{-1}\}\), and the contribution of each such \(\chi\) equals \(g_{\Q_{(n)}}(\ell)\) (with the usual doubling in the exceptional supersingular congruence).

Because \(\beta(\ell)\in\F_p^\times\) is a fixed scalar, multiplication by \(\beta(\ell)^{-1}\) gives a bijection on the character values at \(\ell\):
\(\chi(\ell)\mapsto\beta(\ell)^{-1}\chi(\ell)\). Hence counting characters with \(\chi(\ell)=\beta(\ell)\alpha^{-1}\) is equivalent to counting those with value \(\alpha^{-1}\); the counts are unchanged. By the character-counting result for abelian extensions (see \cite[Th.~3.7]{washington2012introduction}), if \(f\) is the residue degree of \(\ell\) in \(K/\Q\) then
\[
|\{\chi:\zeta_\chi=\alpha\}|=
\begin{cases}
g_K(\ell)&\text{if }\alpha^f=1,\\
0&\text{otherwise},
\end{cases}
\]
and similarly for the other root. Using \(p\nmid[K:\Q]\) so that \(g_{K_{(n)}}(\ell)=g_{\Q_{(n)}}(\ell)g_K(\ell)\), we obtain the first displayed case.

If \(\ell\mid N\) the defining congruence for a nonzero \(g_{\chi,n}(\ell)\) is simply \(a_\ell(h)\equiv\chi(\ell)\pmod\pi\); put \(\alpha=a_\ell(h)\bmod p\). The same character-counting gives the second displayed case, and all remaining cases give zero by definition. This completes the proof.
\end{proof}

\begin{lemma}\label{lemma: reduction to prime to p ext degree}
    Let $M, K$ and $L$ be abelian extensions of $\Q$ satisfying \textbf{(Add)} and $M\subset K \subset L$ with $[L:M]$ a power of $p$. Further, assume that $p\nmid [M:\Q]$ and $n\geq \op{ord}_p([L:\Q])$. If the assertion of Theorem \ref{thm: kida formula!} holds for the extensions $L/M$ and $K/M$, it also holds for $L/K$. 
\end{lemma}
\begin{proof}
Note that $p\nmid[M:\Q]$ implies $M\cap \Q_{\infty}=\Q$, so $n_{M}=0$. For the $\mu$-invariant, suppose that $\mu(\Theta_{n+n_L-n_K}(E/K))=0$. By Remark \ref{remark: sum decomp for iwasawa invs f/K}, this implies $\mu^{M_\psi}_{n+n_L}(\psi)=0$ for all $\psi \in \widehat{\op{Gal}(K/\mathbb{Q})}$. This implies $\mu^{M_\psi}_{n+n_L}(\psi)=0$ for $\psi \in \widehat{\op{Gal}(M/\mathbb{Q})}$, and hence $\mu(\Theta_{n+n_L}(E/M))=0$. (recall, $M_\psi$ denotes the prime-to-$p$ conductor of $\psi$, not to be confused with the extension $M$). Then, from Theorem \ref{thm: kida formula!} for $L/M$, it follows that $\mu(\Theta_{n}(E/L))=0$.
For the $\lambda$-invariant: if Theorem \ref{thm: kida formula!} holds for $L/M$ and $K/M$, we have for all $n\geq 1$:
\begin{equation}\label{eq: L to M}
    \begin{split}
    \lambda(\Theta_{n}(E/L))&=[L_\infty:M_\infty]\lambda(\Theta_{n+n_L}(E/M))+ \sum_{w\in P_1} (e_{L_{(n)}/M_{(n+n_L)}}(w)-1)\\ &\qquad +2\sum_{w\in P_2}(e_{L_{(n)}/M_{(n+n_L)}}(w)-1) \text{ and }
\end{split}
\end{equation}
\begin{equation}
    \begin{split}\label{eq: K to M}
     \lambda(\Theta_{n}(E/K))&=[K_\infty:M_\infty]\lambda(\Theta_{n+n_K}(E/M))+ \sum_{w\in P_1} (e_{K_{(n)}/M_{(n+n_K)}}(w)-1)\\ &\qquad +2\sum_{w\in P_2} (e_{K_{(n)}/M_{(n+n_K)}}(w)-1).
\end{split}
\end{equation}
Note that $f_{L_{(n)}/M_{(n+n_L)}}(w)=1$ and $f_{K_{(n)}/M_{(n+n_K)}}(w)=1$ since $n\geq \op{ord}_p([L:\Q])$.
Since \[[L_\infty:M_\infty]=[L_\infty:K_\infty][K_\infty:M_\infty],\] from \eqref{eq: K to M} we obtain 
\begin{equation}\label{eq: L to K}
    \begin{split}
[L_\infty:K_\infty]\lambda(\Theta_{n+n_L-n_K}(E/K))=&[L_\infty:M_\infty]\lambda(\Theta_{n+n_L}(E/M))\\+& [L_\infty:K_\infty]\sum_{w\in P_1} (e_{K_{(n+n_L-n_K)}/M_{(n+n_L)}}(w)-1)\\  +&2[L_\infty:K_\infty]\sum_{w\in P_2} (e_{K_{(n+n_L-n_K)}/M_{(n+n_L)}}(w)-1).
\end{split}
\end{equation}
Let $v$ be a prime of $K_{(n+n_L-n_K)}$ not lying above $p$ and let $g$ denote the number of primes of $L_{(n)}$ lying above $v$. For any prime $w$ of $L_{(n)}$ lying above $v$, we have \[[L_\infty:K_\infty]=[L_{(n)}:K_{(n+n_L-n_K)}]=e_{L_{(n)}/K_{(n+n_L-n_K)}}(w)g.\] 
This gives us 
\[\begin{split}
&[L_\infty:K_\infty](e_{K_{(n+n_L-n_K)}/M_{(n+n_L)}}(v)-1)\\= &\sum_w (e_{L_{(n)}/M_{(n+n_L)}}(w)-1) -\sum_w (e_{L_{(n)}/K_{(n+n_L-n_K)}}(w)-1)
\end{split}\]
where the sum is over all primes $w$ of $L_{(n)}$ lying over $v$. Then, substituting the expression for $\lambda(\Theta_{n+n_L}(E/M))$ from Equation \eqref{eq: L to K} into Equation \eqref{eq: L to M} gives the desired result.
\end{proof}

\begin{proof}[Proof of Theorem \ref{thm: kida formula!}]
From Lemma \ref{lemma: reduction to prime to p ext degree}, it suffices to prove Theorem \ref{thm: kida formula!} for the case when $p\nmid [K:\Q]$. Let $L'$ denote the maximal subfield of $L$ such that $[L':\Q]$ is a power of $p$. By the assumption $p\nmid[K:\Q]$, we have $n_K=0$ and $K\cap L'=\Q$. Thus, any character on $\op{Gal}(L/\Q)$ decomposes uniquely as a product of a character $\psi$ on $\op{Gal}(K/\Q)$ and a character $\chi$ on $\op{Gal}(L'/\Q)$. 
\begin{figure}[h]
    \centering
\[
\begin{tikzcd}[column sep=large, row sep=large]
    {L_{(n)}} \\
    & L \\
    {K_{(n+n_L)}} && {L'} \\
    & K \\
    && {\mathbb{Q}}
    \arrow[dashed, no head, from=2-2, to=1-1]
    \arrow[no head, from=3-1, to=1-1]
    \arrow[no head, from=3-3, to=2-2]
    \arrow["\chi"{description}, bend left=35, no head, from=3-3, to=5-3]
    \arrow[no head, from=4-2, to=2-2]
    \arrow[dashed, no head, from=4-2, to=3-1]
    \arrow[no head, from=5-3, to=3-3]
    \arrow[no head, from=5-3, to=4-2]
    \arrow["{\psi}"{description, pos=0.6}, bend left=25, no head, from=5-3, to=4-2]
\end{tikzcd}
\]
 \caption{}
    \label{fig:placeholder2}
\end{figure}
For a character $\kappa$, let $M_{\kappa}$ denote the non-$p$ part of its conductor, i.e., the conductor of $\kappa$ is of the form $p^{n'}M_\kappa$ for some $n'\geq0$ with $(p,M_\kappa)=1$.

Then, we have
    \begin{align*}
        p^{n_L}\lambda(\Theta_{n}(E/L))&= \sum_\chi\sum_\psi \lambda(\Theta^{M_{\chi\psi}}_{n+n_L}(\chi\psi)), \quad \lambda(\Theta_{n+n_L}(E/K))=\sum_{\psi}\lambda(\Theta_{n+n_L}^{M_\psi}(\psi))\\
        \mu(\Theta_n(E/L))&= \sum_\chi\sum_\psi \mu(\Theta^{M_{\chi\psi}}_{n+n_L}(\chi\psi)),  \quad \mu(\Theta_{n+n_L}(E/K))=\sum_{\psi}\mu(\Theta_{n+n_L}^{M_{\psi}}(\psi))
    \end{align*}
Since we have assumed $\mu(\Theta_n(E/K))=0$, we have $\mu(\Theta_{n+n_L}^{M_{\psi}}(\psi))=0$ for all $\psi$. Note that the order of $\psi$ is co-prime to $p$ and $\chi$ has $p$-power order, so $M_{\psi}\mid M_{\chi\psi}$ and $M_{\chi\psi}/M_\psi$ is squarefree (a character of $p$-power order must have squarefree prime-to-$p$ conductor). Thus, using Proposition \ref{prop: lambda inv from euler factor}, we have $\mu(\Theta_{n+n_L}^{M_{\chi\psi}}(\chi\psi))=0$ ($\chi$ has $p$-power order) and thus $\mu(\Theta_{n+n_L}(E/L))=0$.
For the $\lambda$-invariants, Proposition \ref{prop: lambda inv from euler factor} gives us
 \[\lambda(\Theta^{M_{\chi\psi}}_{n+n_L}(\chi\psi))= \lambda(\Theta^{M_{\chi\psi}}_{n+n_L}(\psi))= \lambda(\Theta^{M_{\psi}}_{n+n_L}(\psi))+ \sum_{\ell \mid M_{\chi\psi}/M_\psi} g_{\psi,n+n_L}(\ell)\]   
Note that if $\ell\mid M_{\psi}$, then $g_{\psi,n+n_L}(\ell)=0$ (as $\psi(\ell)=0$ for such $\ell$). Thus, we can take the sum over $\ell \mid M_{\chi\psi}$. We need to simplify the sum
\[\sum_{\chi}\sum_\psi\sum_{\ell \mid M_{\chi\psi}} g_{\psi,n+n_L}(\ell)\]
to get the final result.
Now if $\ell$ divides $M_{\chi\psi}$, then $\chi\psi(\ell)=0$. If $\chi(\ell)\neq0$ and $\ell\mid M_{\chi\psi}$, we must have $\psi(\ell)=0$, which in turn implies $g_{\psi,n+n_L}(\ell)=0$. So we only need to sum over those $\chi$ with $\chi(\ell)=0$. Furthermore, the outer summation can be taken over all $\ell\neq p$, since $\chi(\ell)\neq0$ for $\ell\nmid M_{\chi\psi}$. So we have
\begin{align*}
    \sum_{\ell \mid M_{\chi\psi}} g_{\psi,n+n_L}(\ell)&= \sum_{\ell\mid M_{\chi\psi}}\sum_\chi (\sum_\psi g_{\psi,n+n_L}(\ell))\\
    &=\sum_{\ell\neq p} |\{\chi: \chi(\ell)=0\}|\sum_\psi g_{\psi,n+n_L}(\ell).
\end{align*}
Putting all of this together, we have
\[p^{n_L}\lambda(\Theta_{n+n_L}(E/L))=[L':\Q]\lambda(\Theta_{n+n_L}(E/K))+ \sum_{\ell\neq p} |\{\chi: \chi(\ell)=0\}| \sum_{\psi}g_{\psi}(\ell).\]
From $L'\cap K=\Q$ and $L\cap K_{(n+n_L)}=K_{(n_L)}$ we get $[L':\Q]=p^{n_L}[L_{(n)}:K_{(n+n_L)}]=p^{n_L}[L_\infty:K_\infty]$. The prime $\ell\neq p$ is unramified in $\Q_\infty/\Q$, so $e_{L'/\Q}(\ell)=e_{L_{(n)}/K_{(n+n_L)}}(w)$ for any prime $w$ of $L_{(n)}$ lying above $\ell$. 
Then, we use \cite[Theorem 3.7]{washington2012introduction}, which implies that 
\[|\{\chi\in \widehat{\op{Gal}(L'/\Q)}: \chi(\ell)\neq 0\}| = [L':\Q]e_{L'/\Q}(\ell)^{-1}\] to get
\begin{align*}
    |\{\chi: \chi(\ell)=0\}|=&[L':\Q](1-e_{L'/\Q}(\ell)^{-1})\\
    =& p^{n_L}[L_\infty:K_\infty](1-e_{L_{(n)}/K_{(n+n_L)}}(w)^{-1})
\end{align*}
and this is clearly $0$ if $\ell$ does not ramify in $L'/\Q$. Observe that if $w$ is a prime of $L_{(n)}$ lying above $\ell$, then the reduction type of $E/L_{(n)}$ at $w$ is the same as that of $E/K_{(n+n_L)}$ at the prime of $K_{(n+n_L)}$ lying above $\ell$ since $L_{(n)}$ satisfies our assumption \textbf{(Add)} as $L_{(n)}/L$ is unramified. Then, Proposition \ref{prop: sum over g_chi} gives us $\sum_{\psi}g_{\psi,n+n_L}(\ell)=\delta_wg_{K_{(n+n_L)}}(\ell)$, where
\[\delta_w=\begin{cases}
    2 \text{ if } w\in P_2,\\
    1 \text{ if } w\in P_1,\\
    0 \text{ otherwise.}
\end{cases}\]
(Also note that Proposition \ref{prop: sum over g_chi} can be applied since in a Galois $p$-extension, if a prime $\ell\neq p$ ramifies, then $\ell\equiv1\pmod{p}$.) 
Hence, for each $\ell\neq p$, we have 
\begin{align*}
    |\{\chi: \chi(\ell)=0\}|\sum_\psi g_{\psi,n+n_L}(\ell)=& p^{n_L}[L_\infty:K_\infty]g_{K_{(n+n_L)}}(\ell)\delta_w(1-e_{L_{(n)}/K_{(n+n_L)}}(w)^{-1})\\
    =&p^{n_L}[L_\infty:K_\infty]g_{K_{(n+n_L)}}(\ell)g_{L_{(n)}}(\ell)^{-1}\sum_{w\mid l}\delta_w (1-e_{L_{(n)}/K_{(n+n_L)}}(w)^{-1}),
\end{align*}
where in the second equality, we write $g_{L_{(n)}}(\ell)\delta_w(1-e_{L_{(n)}/K_{(n+n_L)}}(w)^{-1}) = \sum_{w\mid \ell}\delta_w(1-e_{L_{(n)}/K_{(n+n_L)}}(w)^{-1})$ since $g_{L_{(n)}}(\ell)$ is the number of primes of $L_{(n)}$ lying over $\ell$ by definition. Now \[e_{L_{(n)}/K_{(n+n_L)}}(w)f_{L_{(n)}/K_{(n+n_L)}}(w)g_{L_{(n)}}(\ell)=[L_{(n)}:K_{(n+n_L)}]g_{K_{(n+n_L)}}(\ell),\] so we get 
\[\lambda(\Theta_{n}(E/L))=[L_\infty:K_\infty]\lambda(\Theta_{n+n_L}(E/K))+ \sum_{w\nmid p} \delta_wf_{L_{(n)}/K_{(n+n_L)}}(w)(e_{L_{(n)}/K_{(n+n_L)}}(w)-1)\]
which completes the proof.
\end{proof}
Note that the argument applies to the case of Hecke eigenforms $f \in S_k(\Gamma_0(N))$. In that case, one can remove the assumption \textbf{(Add)} since the sets of primes $P_1$ and $P_2$ of $L_{(n)}$ are defined in terms of the rational primes. In particular, for $w \in L_{(n)}$, let $\tilde{w}\in \Q$ denote the rational prime lying below $w$. Define the sets $P_1,P_2$ as 
    \begin{align*}
        P_1&:=\{w : w\nmid p, \tilde{w} \mid N \text{ and } \alpha_{\tilde{w}}^{f_{\tilde{w}}}\not\equiv 1\pmod{p}\},\\
        P_2&:=\{w :w\nmid p, \tilde{w}\nmid N, \text{ and } \alpha_{\tilde{w}}^{f_{\tilde{w}}}\equiv 1\pmod{p}\},
    \end{align*}
where $\alpha_{\tilde{w}}\in \overline{\F}_p$ denotes a root of $x^2-a_\ell(\tilde{w})x+\epsilon_N(\tilde{w})\pmod{p}$ and $f_{\tilde{w}}$ denotes the residue degree of $\tilde{w}$ in $K/\Q$. Thus, applying Proposition \ref{prop: sum g_chi higher weight} to $\sum_\psi g_{\psi, n+n_L}(\ell)$ for a character $\psi$ on $K/\Q$ as in the previous proof gives us the following:
\begin{theorem}\label{thm: modforms Kida formula}
Let $f \in S_k(\Gamma_0(N)$ be a Hecke eigenform. Let $n\geq 1$ and assume that the condition \textbf{(K-p)} is satisfied. If
    \[
        \mu(\Theta_{\,n+n_L-n_K}(f/K))=0,
    \]
    then $\mu(\Theta_n(f/L))=0$, and
    \[
    \begin{aligned}
        \lambda(\Theta_n(f/L)) 
        &= [L_\infty : K_\infty] \, \lambda(\Theta_{\,n+n_L-n_K}(f/K))
        + \sum_{w \in P_1} f_{L_{(n)}/K_{(n+n_L-n_K)}}(w)\bigl(e_{L_{(n)}/K_{(n+n_L-n_K)}}(w)-1\bigr) \\
        &\quad + 2 \sum_{w \in P_2} f_{L_{(n)}/K_{(n+n_L-n_K)}}(w)\bigl(e_{L_{(n)}/K_{(n+n_L-n_K)}}(w)-1\bigr).
    \end{aligned}
    \]
    where the sets $P_1, P_2$ are sets of primes of $L_{(n)}$ defined as above.
\end{theorem}

\section{Applications}\label{s 4}
Let $E$ be an elliptic curve defined over $\Q$ and let $p$ be an odd prime. Following \cite[Page 84]{MATSUNO200080}, we define the $p$-adic $L$-function of $E$ over an abelian extension $K$ by 
\begin{equation}\label{defn of Lp(E/K)}L_p(E/K, (1+T)^{p^{n_K}}-1)\coloneqq \prod_{\psi \in \widehat{\op{Gal}(K/\Q)}} L_p(E,\psi,T)\end{equation}
where $L_p(E,\psi,T)$ denotes the $p$-adic $L$-function for $E$ defined as an integral of the character $\psi$ against a measure on $\Z/M\Z\times\Zp^\times $ constructed using modular symbols in the usual way (see \cite[Page 51]{MSD74}, \cite{MTT} for details). 
 \subsection{Kida's formula for $p$-adic \textit{L}-functions at primes of semistable reduction}
 \par First, we highlight that in the case when $p$ is a prime of good ordinary or multiplicative reduction, Theorem \ref{thm: kida formula!} allows us to recover the Kida-type formula obtained by Matsuno in \cite{MATSUNO200080} for the $\lambda$-invariant of the $p$-adic $L$-function of $E$. To see this, note that for $n$ sufficiently large, we have
\begin{align*}
    \mu(L_p(E/K,T))+\frac{\lambda(L_p(E/K,T))}{p^n-p^{n-1}} &= \mu(\Theta_n(E/K))+\frac{\lambda(\Theta_n(E/K))}{p^n-p^{n-1}}
\end{align*}
using the interpolation property, assuming $\lambda(\Theta_n(E/K))<p^n- p^{n-1}$. 
\begin{remark}
    The condition $\lambda(\Theta_n(E/K))<p^n- p^{n-1}$ holds when $E$ has multiplicative reduction at $p$, or $E$ has good ordinary reduction and $E[p]$ is irreducible as a $\op{Gal}(\overline{\Q}/\Q)$-representation. This is because $\lambda(\Theta_n(E, \psi, T))=\lambda(L_p(E,\psi,T))$ for $n$ sufficiently large by \cite[Proposition 3.7]{PW}. Thus $\lambda(\Theta_n(E/K))$ will also be eventually constant for $n$ sufficiently large. 
\end{remark}
Under the assumption $\mu(\Theta_n(E/K))=0$, it follows that $\mu(L_p(E/K))=0$ and thus $\lambda(\Theta_n(E/K))=\lambda(L_p(E/K))$, which recovers the Kida formula of \cite[Theorem~3.1]{MATSUNO200080}. 

Furthermore, we can apply our result to the case when $p$ is a prime of good supersingular reduction to obtain the following extension of Matsuno's result. We use the decomposition of $L_p(E,\chi,T)$ into the signed $p$-adic $L$-functions $L_p^\pm(E,\chi,T)$ defined by Pollack. See \cite[Theorem~5.6]{pollack03} for the definition of $L_p^\pm(E, \psi,T)$. We define $L_p^\pm(E/K, T)$ as a product over the $L_p^\pm(E, \psi,T)$ as in Equation \ref{defn of Lp(E/K)}.
\begin{theorem}\label{kida for signed p-adic $L$-functions}
Let $E_{/\Q}$ be an elliptic curve and $p$ be an odd prime number for which $a_p(E)=0$. Let $L/K$ be an extension of number fields such that $L/K$ is Galois and $\op{Gal}(L/K)$ is a $p$-group. Suppose $L,K$ both satisfy \textbf{(Add)}. Then, if $\mu(L_p^\pm(E/K))=0$ we have that $\mu(L_p^\pm(E/L))=0$ and
\begin{align*}
\lambda(L_p^\pm(E/L))=[L_\infty:K_\infty]\lambda(L_p^\pm(E/K))+ \sum_{w\in P_1} (e_{L_\infty/K_\infty}(w)-1)\\+ 2\sum_{w\in P_2} (e_{L_\infty/K_\infty}(w)-1)
    \end{align*}
    where $P_1,P_2$ are sets of primes of $L_\infty$ defined as 
    \begin{align*}
        P_1:=\{w : w\nmid p, E \text{ split multiplicative reduction at } w\},\\
        P_2:=\{w : w\nmid p, E \text{ good reduction at } w, E(L_{\infty,w})[p])\neq0\}
    \end{align*}
\end{theorem}
\begin{proof}
    The plus-minus decomposition for $L_p(E,\psi,T)$ is related to the Mazur--Tate element by \cite[Proposition~6.18]{pollack03}, which states
\[\Theta_n(E,\psi,T) \equiv \begin{cases}
     \omega^-_n(T) L_p^-(E,\psi,T) \pmod{((1+T)^{p^n}-1)}\text{ if $n$ even} \\
     \omega^+_n(T) L_p^+(E,\psi,T) \pmod{((1+T)^{p^n}-1)}\text{ if $n$ odd}
\end{cases}\]
    where $\omega^\pm_n(T)$ is defined as a product of cyclotomic polynomials defined above \cite[Proposition 6.18]{pollack03}. Under the assumption $\mu(L_p^\pm(E/K))=0$, we have $\mu(L_p^\pm(E,\psi,T))$ for $\psi \in \widehat{\op{Gal}(K/\Q)}$ (from the product decomposition). This implies $\mu(\Theta_n(E,\psi,T))=0$ and gives us 
    \[\lambda(\Theta_n(E,\psi,T))= q_n + \lambda(L_p^\pm(E,\psi, T))\]
    where 
    \[ q_n=\begin{cases}
    p^{n-1}-p^{n-2}+\cdots+p-1 \space \text{ if $n$ even}\\
    p^{n-1}-p^{n-2}+\cdots+p^2-p \space \text{ if $n$ odd}
\end{cases} \]
since $\mu(\omega^\pm_n(T))=0$ and $\lambda(\omega^\pm_n(T))=q_n$.
(see also \cite[Theorem 4.1]{PW}).
    Then, using the definition for the Mazur--Tate element over an abelian extension $K$, we obtain
     \[\lambda(\Theta_n(E/K))= [K:\Q]q_n+ \lambda(L_p^\pm(E/K)).\]
     Now use $[L_\infty:K_\infty][K:\Q]=[L:\Q]$ and apply Theorem \ref{thm: kida formula!} to get the result, assuming $\mu^\pm(L_p(E/K))=0$.
\end{proof}
\begin{remark}
\begin{enumerate}
    \item Similar to the method above, our methods can be applied to obtain Kida-type formulae for plus-minus $\lambda$-invariants for the various classes of $p$-non-ordinary modular forms studied in \cite{PW}.
    \item In \cite{PWKida}, formulae similar to the one above are proven, both on the algebraic side as well as the analytic side, by using formulae for $\lambda$-invariants of $p$-congruent cusp forms.
\end{enumerate}
\end{remark}
\subsection{Kida's formula for $p$-adic \textit{L}-functions at primes of additive, potentially ordinary reduction}
In \cite{RSadditivekida}, an analogue of Kida's formula was proven for elliptic curves over $\Q$ with additive reduction at a prime $p$ under certain additional hypotheses motivated by the work of Delbourgo in \cite{del-compositio}. In particular, their result applies to the case when $p$ is a prime of additive reduction and $E$ achieves good ordinary reduction over an extension contained in $\Q(\mu_p)$ (this corresponds to Hypothesis (G) of Delbourgo, see \cite[\S~1.5]{del-compositio}). We refer to this condition as the potentially good ordinary condition. Under this condition, Delbourgo defined a $p$-adic $L$-function for $E$. This was done by showing the existence a weight $2$ cuspidal newform $\tilde f$ of level $N_E/p^2$ corresponding to the twist of $f_E$ by a character of conductor $p$ (which can be taken as a power of the Teichm\"uller character $\omega$). Such an $\tilde f$ is ordinary at $p$ and so has a bounded $p$-adic $L$-function attached to it.

Now $f_E\otimes\omega^{i}\sim  \tilde f$ for $i$ depending only on $E$ (i.e., $a_n(E)\omega^i(n)=a_n(\tilde f)$ for all $(n,N_E)=1$), so one defines 
\[L_p^{\op{Del}}(E,\psi,T):= \mathfrak{L}_p\cdot L_p(\tilde{f},\omega^{-i}\psi,T)\]
up to certain $p$-adic multiplier $\mathfrak{L}_p$, see \cite[\S~1.5]{del-compositio} for details. From \cite[Theorem B]{lei2025iwasawainvariantsmazurtateelements}, we have the following equality of Iwasawa invariants of Mazur--Tate elements at additive primes $p>3$ for $n$ sufficiently large: 
\begin{align*}
    \mu(\Theta_n(E,\psi,T))&=\mu(L_p^{\op{Del}}(E,\psi,T))\\
    \lambda(\Theta_n(E,\psi,T))&=\frac{\ord_p(\Delta_E)(p-1)}{12}p^{n-1}+ \lambda(L_p^{\op{Del}}(E,\psi,T))
\end{align*}
where $\Delta_E$ denotes the minimal discriminant of $E/\Q$. This implies 
\begin{equation}\label{eq: lambda additive}
    \lambda(\Theta_n(E/K))= [K:\Q]\frac{\ord_p(\Delta_E)(p-1)}{12}p^{n-1}+\lambda(L_p^{\op{Del}}(E/K,T))
\end{equation}
 Note that \cite[Theorem B]{lei2025iwasawainvariantsmazurtateelements} requires the assumption of the Birch and Swinnerton-Dyer conjecture to express the Iwasawa invariants of the Mazur--Tate elements in terms of the algebraic Iwasawa invariants. For our purposes, such an assumption is not required since we only relate the Iwasawa invariants of Mazur--Tate elements to those of the $p$-adic $L$-function. This is done by comparing the interpolation property of $\Theta_n(E,\psi,T)$ with those of $L_p(\tilde{f},\omega^{-i}\psi,T)$.
\begin{remark}
    The power $i$ is related to the ``semistability defect'' of $E$, which is determined by the degree of the smallest extension over which $E$ achieves good ordinary reduction at $p$.
\end{remark}

\begin{theorem}\label{additive reduction kida formula}
 Let $L/K$ be a $p$-extension of number fields with $L$ and $K$ abelian over $\Q$ and suppose $L,K$ both satisfy \textbf{(Add)}. Let $E$ be an elliptic curve defined over $\Q$ with additive, potentially ordinary reduction at an odd prime $p>3$. Then,
\begin{align*}
    \lambda(L_p^{\op{Del}}(E/L))=[L_\infty:K_\infty]\lambda(L_p^{\op{Del}}(E/K))+ \sum_{w\in P_1} (e_{L_\infty/K_\infty}(w)-1)\\+ 2\sum_{w\in P_2} (e_{L_\infty/K_\infty}(w)-1)
    \end{align*}
    where $P_1,P_2$ are sets of primes of $L_\infty$ defined as follows:
    \begin{align*}
        & P_1:=\{w : w\nmid p, E \text{ split multiplicative reduction at } w\},\\
        & P_2:=\{w : w\nmid p, E \text{ good reduction at } w, E(L_{\infty,w})[p])\neq0\}.
    \end{align*}  
\end{theorem}
\begin{proof}
Apply Theorem \ref{thm: kida formula!} and take $n$ sufficiently large to obtain a formula relating the Iwasawa invariants of $\Theta_n(E/L)$ and $\Theta_n(E/K)$. Note that $L_{(n)}=L\cdot \Q_{(n+n_L)}$ and \[K_{(n+n_L-n_K)}=K\cdot \Q_{(n+n_L-n_K+n_K)}=K\cdot \Q_{(n+n_L)}.\] Hence, 
\[[L_{(n)}: K_{(n+n_L-n_K)}]= [L\cdot \Q_{(n+n_L)}: K\cdot \Q_{(n+n_L)}]=[L_\infty:K_\infty]\] for any $n\geq 0$. Then, using Equation \eqref{eq: lambda additive} gives us the desired result as the growth terms cancel out.
\end{proof}
\noindent The above result implies that if the Iwasawa Main conjecture formulated for these elliptic curves by Delbourgo in \cite{DELBOURGO200238} holds over $\Q$, then 
since the algebraic and analytic Iwasawa invariants have the same expression (by \cite[Theorem B]{RSadditivekida} and Theorem \ref{additive reduction kida formula}) when we base-change to a $p$-extension of $\Q$, a one-sided divisibility result for the base-change would imply the full Main conjecture over that extension.
\bibliographystyle{amsalpha}
\bibliography{references}

@article {Iwasawa:1959,
    AUTHOR = {Iwasawa, Kenkichi},
     TITLE = {On {$\Gamma $}-extensions of algebraic number fields},
   JOURNAL = {Bulletin of the American Mathematical Society},
  FJOURNAL = {Bulletin of the American Mathematical Society},
    VOLUME = {65},
      YEAR = {1959},
     PAGES = {183--226},
}

@article{iwasawa1973zl,
  title={On $\mathbb{Z}_\ell$-extensions of algebraic number fields},
  author={Iwasawa, Kenkichi},
  journal={Annals of Mathematics},
  pages={246--326},
  year={1973},
  publisher={JSTOR}
}

@article {pollack03,
    AUTHOR = {Pollack, Robert},
     TITLE = {On the {$p$}-adic {$L$}-function of a modular form at a
              supersingular prime},
   JOURNAL = {Duke Mathematical Journal},
  FJOURNAL = {Duke Mathematical Journal},
    VOLUME = {118},
      YEAR = {2003},
    NUMBER = {3},
     PAGES = {523--558},
      }

@article {del-compositio,
    AUTHOR = {Delbourgo, Daniel},
     TITLE = {Iwasawa theory for elliptic curves at unstable primes},
   JOURNAL = {Compositio Math.},
  FJOURNAL = {Compositio Mathematica},
    VOLUME = {113},
      YEAR = {1998},
    NUMBER = {2},
     PAGES = {123--153},
      
}

@article {pollack05,
    AUTHOR = {Pollack, Robert},
     TITLE = {An algebraic version of a theorem of {K}urihara},
   JOURNAL = {Journal of Number Theory},
  FJOURNAL = {Journal of Number Theory},
    VOLUME = {110},
      YEAR = {2005},
    NUMBER = {1},
     PAGES = {164--177},
      
}

@article {PW,
    AUTHOR = {Pollack, Robert and Weston, Tom},
     TITLE = {Mazur-{T}ate elements of nonordinary modular forms},
   JOURNAL = {Duke Mathematical Journal},
  FJOURNAL = {Duke Mathematical Journal},
    VOLUME = {156},
      YEAR = {2011},
    NUMBER = {3},
     PAGES = {349--385},
      
}

@article {MT,
    AUTHOR = {Mazur, Barry and Tate, John},
     TITLE = {Refined conjectures of the ``{B}irch and {S}winnerton-{D}yer
              type''},
   JOURNAL = {Duke Mathematical Journal},
  FJOURNAL = {Duke Mathematical Journal},
    VOLUME = {54},
      YEAR = {1987},
    NUMBER = {2},
     PAGES = {711--750},
    
}

@article {forrasmuller,
    AUTHOR = {Forr\'{a}s, Ben and M\"{u}ller, Katharina},
     TITLE = {On the cohomology of plus/minus {S}elmer groups of
              supersingular elliptic curves in weakly ramified base fields},
   JOURNAL = {Research in Number Theory},
  FJOURNAL = {Research in Number Theory},
    VOLUME = {11},
      YEAR = {2025},
    NUMBER = {2},
     PAGES = {Paper No. 56, 35},
      ISSN = {2522-0160},
}

@article {hatleylei,
    AUTHOR = {Hatley, Jeffrey and Lei, Antonio},
     TITLE = {Arithmetic properties of signed {S}elmer groups at
              non-ordinary primes},
   JOURNAL = {Annales de l'Institut Fourier (Grenoble)},
  FJOURNAL = {Universit\'{e} de Grenoble. Annales de l'Institut Fourier},
    VOLUME = {69},
      YEAR = {2019},
    NUMBER = {3},
     PAGES = {1259--1294},
      ISSN = {0373-0956},
}

@ARTICLE{MSD74,
  author = {Mazur, Barry and Swinnerton-Dyer, Peter},
  title = {Arithmetic of {W}eil curves},
  journal = {Inventiones Mathematicae},
  year = {1974},
  volume = {25},
  pages = {1--61},
  doi = {10.1007/BF01389997},
  fjournal = {Inventiones Mathematicae}
}

@ARTICLE{MTT,
  author = {Mazur, Barry and Tate, John and Teitelbaum, Jeremy},
  title = {On {$p$}-adic analogues of the conjectures of {B}irch and {S}winnerton-{D}yer},
  journal = {Inventiones Mathematicae},
  year = {1986},
  volume = {84},
  pages = {1--48},
  number = {1},
  fjournal = {Inventiones Mathematicae}
}

@article{DELBOURGO200238,
title = {On the p-Adic {B}irch, {S}winnerton–{D}yer {C}onjecture for Non-semistable Reduction},
journal = {Journal of Number Theory},
volume = {95},
number = {1},
pages = {38-71},
year = {2002},
issn = {0022-314X},
doi = {https://doi.org/10.1006/jnth.2001.2755},
url = {https://www.sciencedirect.com/science/article/pii/S0022314X01927556},
author = {Daniel Delbourgo},
}

@article{Mazur1972,
author = {Mazur, Barry},
journal = {Inventiones mathematicae},
pages = {183-266},
title = {Rational Points of Abelian Varieties with Values in Towers of Number Fields.},
url = {http://eudml.org/doc/142180},
volume = {18},
year = {1972},
}

@incollection{kato1,
     author = {Kato, Kazuya},
     title = {$p$-adic {Hodge} theory and values of zeta functions of modular forms},
     booktitle = {Cohomologie $p$-adiques et applications arithm\'etiques (III)},
     series = {Ast\'erisque},
     pages = {117--290},
     publisher = {Soci\'et\'e math\'ematique de France},
     number = {295},
     year = {2004},
}

@article {Rubin2,
    AUTHOR = {Rubin, Karl},
     TITLE = {On the {M}ain conjecture of {I}wasawa theory for imaginary
              quadratic fields},
   JOURNAL = {Inventiones Mathematicae},
  FJOURNAL = {Inventiones Mathematicae},
    VOLUME = {93},
      YEAR = {1988},
    NUMBER = {3},
     PAGES = {701--713},
}

@article{kobayashi,
author = {Kobayashi, Shin-ichi},
year = {2003},
month = {04},
pages = {1-36},
title = {Iwasawa theory for elliptic curves at supersingular primes},
volume = {152},
journal = {Inventiones Mathematicae},
doi = {10.1007/s00222-002-0265-4}
}

@misc{lei2025iwasawainvariantsmazurtateelements,
      title={On the {I}wasawa Invariants of {M}azur--{T}ate elements of elliptic curves at additive primes https://arxiv.org/abs/2412.16629}, 
      author={Lei, Antonio and Pollack, Robert and Pratap, Naman},
      year={2025},
      eprint={2412.16629},
      archivePrefix={arXiv},
      primaryClass={math.NT},
}

@book{washington2012introduction,
  title={Introduction to Cyclotomic Fields},
  author={Washington, L.C.},
  isbn={9781461219347},
  lccn={96013169},
  series={Graduate Texts in Mathematics},
  url={https://books.google.co.in/books?id=27zkBwAAQBAJ},
  year={1997},
  publisher={Springer New York}
}

@article {shimura,
    AUTHOR = {Shimura, Goro},
     TITLE = {The special values of the zeta functions associated with cusp
              forms},
   JOURNAL = {Communications on Pure and Applied Mathematics},
  FJOURNAL = {Communications on Pure and Applied Mathematics},
    VOLUME = {29},
      YEAR = {1976},
    NUMBER = {6},
     PAGES = {783--804},
}

@article {AshStevens,
    AUTHOR = {Ash, Avner and Stevens, Glenn},
     TITLE = {Modular forms in characteristic {$l$} and special values of
              their {$L$}-functions},
   JOURNAL = {Duke Mathematical Journal},
  FJOURNAL = {Duke Mathematical Journal},
    VOLUME = {53},
      YEAR = {1986},
    NUMBER = {3},
     PAGES = {849--868},
}

@article {PacsolPopa,
    AUTHOR = {Pa\c{s}ol, Vicen\c{t}iu and Popa, Alexandru A.},
     TITLE = {Modular forms and period polynomials},
   JOURNAL = {Proceedings of the London Mathematical Society},
  FJOURNAL = {Proceedings of the London Mathematical Society. Third Series},
    VOLUME = {107},
      YEAR = {2013},
    NUMBER = {4},
     PAGES = {713--743},
}

@inproceedings {Tate,
    AUTHOR = {Tate, J.},
     TITLE = {Number theoretic background},
 BOOKTITLE = {Automorphic forms, representations and {$L$}-functions},
    SERIES = {Proceedings Symposium of Pure Math., XXXIII},
     PAGES = {3--26},
 PUBLISHER = {Amer. Math. Soc., Providence, RI},
      YEAR = {1979},
}

@article{kuri2002,
title = "On the {T}ate--{S}hafarevich groups over cyclotomic fields of an elliptic curve with supersingular reduction {I}",
author = "Masato Kurihara",
year = "2002",
doi = "10.1007/s002220100206",
volume = "149",
pages = "195--224",
journal = "Inventiones Mathematicae",
issn = "0020-9910",
publisher = "Springer New York",
number = "1",
}

@article{MATSUNO200080,
title = {An Analogue of {K}ida's Formula for the p-adic {L}-functions of Modular Elliptic Curves},
journal = {Journal of Number Theory},
volume = {84},
number = {1},
pages = {80-92},
year = {2000},
}

@article {RSadditivekida,
    AUTHOR = {Ray, Anwesh and Shingavekar, Pratiksha},
     TITLE = {An analogue of {K}ida's formula for elliptic curves with
              additive reduction},
   JOURNAL = {Ramanujan J.},
  FJOURNAL = {Ramanujan Journal. An International Journal Devoted to the
              Areas of Mathematics Influenced by Ramanujan},
    VOLUME = {65},
      YEAR = {2024},
    NUMBER = {2},
     PAGES = {857--883},
      ISSN = {1382-4090,1572-9303},
   MRCLASS = {11R23 (11G05 11G07)},
  MRNUMBER = {4800994},
MRREVIEWER = {Francesc\ Bars},
       DOI = {10.1007/s11139-024-00920-8},
       URL = {https://doi.org/10.1007/s11139-024-00920-8},
}

@article {PWKida,
    AUTHOR = {Pollack, Robert and Weston, Tom},
     TITLE = {Kida's formula and congruences},
   JOURNAL = {Documenta Mathematic},
  FJOURNAL = {Documenta Mathematica},
      YEAR = {2006},
     PAGES = {615--630},
      ISSN = {1431-0635,1431-0643},
}

@article{KIDA1980519,
title = {{$\ell$}-Extensions of {CM}-fields and cyclotomic invariants},
journal = {Journal of Number Theory},
volume = {12},
number = {4},
pages = {519-528},
year = {1980},
issn = {0022-314X},
doi = {https://doi.org/10.1016/0022-314X(80)90042-6},
url = {https://www.sciencedirect.com/science/article/pii/0022314X80900426},
author = {Yûji Kida}
}

@article{Sinnott,
     author = {Sinnott, Warren M.},
     title = {On $p$-adic {L}-functions and the {Riemann-Hurwitz} genus formula},
     journal = {Compositio Mathematica},
     pages = {3--17},
     publisher = {Martinus Nijhoff Publishers},
     volume = {53},
     number = {1},
     year = {1984},
     mrnumber = {762305},
     zbl = {0545.12011},
     language = {en},
     url = {https://www.numdam.org/item/CM_1984__53_1_3_0/}
}

@article{Gras,
     author = {Gras, Georges},
     title = {Sur les invariants $\lambda$-{I}wasawa des corps ab\'eliens},
     journal = {Publications math\'ematiques de la Facult\'e des sciences de Besan\c{c}on. Th\'eorie des nombres},
     eid = {5},
     pages = {1--38},
     publisher = {Universit\'e de Besan\c{c}on},
     year = {1979},
     doi = {10.5802/pmb.a-22},
     language = {fr},
     url = {https://www.numdam.org/articles/10.5802/pmb.a-22/}
}

@ARTICLE{Hachimori1999581,
	author = {Hachimori, Yoshitaka and Matsuno, Kazuo},
	title = {An analogue of {K}ida's formula for the {S}elmer groups of elliptic curves},
	year = {1999},
	journal = {Journal of Algebraic Geometry},
	volume = {8},
	number = {3},
	pages = {581 – 601},
}

@article{kimkurihara,
    author = {Kim, Chan-Ho and Kurihara, Masato},
    title = {On the {R}efined {C}onjectures on {F}itting {I}deals of {S}elmer {G}roups of Elliptic Curves with Supersingular Reduction},
    journal = {International Mathematics Research Notices},
    volume = {2021},
    number = {14},
    pages = {10559-10599},
    year = {2021},
    month = {07},
}
\end{document}